\documentclass[11pt,final]{article}%
\usepackage{stix2}
\usepackage{appendix}
\usepackage{amsfonts}
\usepackage{amsmath}
\usepackage{amssymb}
\usepackage{amsxtra,amscd}
\usepackage[amsmath,hyperref,thmmarks]{ntheorem}
\usepackage{makeidx}
\usepackage{graphicx}
\usepackage{cprotect}
\usepackage[colorlinks=true,bookmarksnumbered=true,pdfpagemode=None,final]%
{hyperref}%
\providecommand{\U}[1]{\protect\rule{.1in}{.1in}}
\theoremnumbering{arabic}
\theoremheaderfont{\scshape}
\RequirePackage{latexsym}
\theorembodyfont{\slshape}
\theoremseparator{.}
\newtheorem{X}{X}[section]

\newtheorem{corollary}[X]{Corollary}
\newtheorem{E}[X]{}
\newtheorem{lemma}[X]{Lemma}

\newtheorem{proposition}[X]{Proposition}

\newtheorem{theorem}[X]{Theorem}

\theorembodyfont{\upshape}
\newtheorem{definition}[X]{Definition}
\newtheorem{example}[X]{Example}

\newtheorem{notation}[X]{Notation}

\newtheorem{remark}[X]{Remark}

\theorembodyfont{\small}
\newtheorem{aside}[X]{Aside}
\newtheorem*{notes}{Notes}
\theorembodyfont{\normalsize}
\theoremstyle{nonumberplain}
\theoremsymbol{\ensuremath{_\Box}}
\newtheorem{proof}{Proof}

\qedsymbol{\ensuremath{_\Box}}

\def\mapsto{\DOTSB\mathchar"39AD }

\usepackage{wrapfig}

\newcommand{\df}{\smash{\lower.12em\hbox{\textup{\tiny def}}}}

\usepackage{changepage}

\let\quoteOld\quote
\let\endquoteOld\endquote
\renewenvironment{quote}{\small\quoteOld}{\endquoteOld}

\usepackage[overload]{textcase}

\usepackage{xcolor}
\definecolor{bblue}{rgb}{0.0, 0.0, 0.6}

\usepackage{array}

\usepackage{microtype}

\usepackage{tikz}
\usepackage{tikz-cd}
\usetikzlibrary{matrix,arrows,positioning,decorations.pathmorphing}
\usetikzlibrary{decorations.markings,shapes.geometric}
\tikzset{commutative diagrams/column sep/Huge/.initial=12ex}
\tikzcdset{arrow style=math font}

\usepackage{wrapfig}

\usepackage[shortlabels]{enumitem}
\setitemize[1]{label=$\diamond$\hspace{0.07in}}
\setlist{topsep=0.1em,itemsep=0.1em,parsep=0.1em}

\setdescription{font=\normalfont}

\usepackage{url}
\usepackage{verbatim}

\usepackage[notcite,notref]{showkeys}

\usepackage{mathtools}

\newcommand{\eb}[1]{{\itshape\bfseries#1}}
\renewcommand{\emph}{\eb}

\usepackage{titletoc,titlesec}
\contentsmargin{2.0em}
\dottedcontents{section}[2.3em]{}{1.8em}{1pc}
\usepackage[nottoc]{tocbibind}

\titleformat*{\section}{\LARGE\bfseries}
\titleformat*{\subsection}{\Large\itshape}
\titleformat*{\subsubsection}{\scshape}
\titleformat*{\paragraph}{\itshape}

\usepackage{natbib}
\let\cite\citealt

\hypersetup{linkcolor=bblue,anchorcolor=bblue,citecolor=bblue,filecolor=bblue,urlcolor=bblue}
\hypersetup{pdftitle={Shimura Varieties},pdfauthor={J.S. Milne}}
\hypersetup{pdfsubject={none},pdfkeywords={none}}

\usepackage[papersize={6.5in,10.0in},margin=0.5in,pdftex]{geometry} 

\newcommand\babstract{\begin{abstract}}\newcommand\eabstract{\end{abstract}}
\newcommand{\bcomment}{}
\newcommand{\bfootnotesize}{\begin{footnotesize}}\newcommand\efootnotesize{\end{footnotesize}}
\newcommand{\bquote}{\begin{quote}}\newcommand\equote{\end{quote}}
\newcommand{\bsmall}{\begin{small}}\newcommand\esmall{\end{small}}
\newcommand{\btable}{\begin{table}}\newcommand{\etable}{\end{table}}

\newcommand{\edocument}{\end{document}}

\newcommand{\tableoc}{\tableofcontents}

\newcommand{\tstyle}{\textstyle}

\renewcommand{\Gamma}{\varGamma}
\renewcommand{\Pi}{\varPi}
\renewcommand{\Sigma}{\varSigma}

\def\1{{1\mkern-7mu1}}

\DeclareMathOperator{\Aut}{Aut}

\DeclareMathOperator{\Div}{Div}

\DeclareMathOperator{\End}{End}

\DeclareMathOperator{\Gal}{Gal}

\DeclareMathOperator{\Hom}{Hom}
\DeclareMathOperator{\id}{id}

\DeclareMathOperator{\Mor}{Mor}

\DeclareMathOperator{\pr}{pr}

\DeclareMathOperator{\Spec}{Spec}

\DeclareMathOperator{\spm}{spm}

\newcommand{\Mod}{\mathsf{Mod}}

\begin{document}

\title{Descent for algebraic schemes}
\date{June 8, 2024}
\author{J.S. Milne}
\maketitle

\babstract This is an elementary exposition of the basic descent theorems for
algebraic schemes over fields (Grothendieck, Weil,\ldots).\eabstract

\renewcommand{\thefootnote}{\fnsymbol{footnote}} \footnotetext{This is a
revised version of Chapter 16 of my notes \textit{Algebraic Geometry}. I've
posted it on the arXiv in order to have a convenient reference.} \renewcommand{\thefootnote}{\arabic{footnote}}

\tableoc

\label{dt}

\bigskip Consider fields $k\subset\Omega$. An algebraic scheme $V$ over $k$
defines a scheme $V_{\Omega}$ over $\Omega$ by extension of the base field.
Descent theory provides answers to the following question: what additional
structure do we need to place on an algebraic scheme over $\Omega$, or a
morphism of algebraic schemes over $\Omega$, in order to ensure that it comes
from $k$? We are most interested in the case that $\Omega$ is algebraically
closed and $k$ is perfect.

In this article, we shall make free use of the axiom of choice (usually in the
form of Zorn's lemma).

\setcounter{section}{-1}

\section{Preliminaries from algebraic geometry}

Let $k$ be a field. An algebraic scheme over $k$ is a separated scheme of
finite type over $\Spec k$. It is integral if it is reduced and irreducible,
and it is an algebraic variety if it is geometrically reduced. An affine
$k$-algebra is a finitely generated $k$-algebra $A$ such that $K\otimes_{k}A$
is reduced for all fields $K$ containing $k$ (it suffices to check this for an
algebraic closure of $k$). A regular map of algebraic schemes (or varieties)
over $k$ is a $k$-morphism. For an affine algebraic scheme $V$ over $k$,
$k[V]=\mathcal{O}{}_{V}(V)$ (so $V=\Spec k[V]$), and for an integral algebraic
scheme $V$ over $k$, $k(V)$ is the field of rational functions on $V$ (the
local ring at the generic point of $V$).

\begin{E}
\label{dt01}In an algebraic scheme $V$ over $k$, the intersection of any two
open affine subsets is again an open affine subset.
\end{E}

Let $U$ and $U^{\prime}$ be open affine subsets of $V$. Then $U\cap U^{\prime
}$ is certainly open, and the diagonal map $U\cap U^{\prime}\hookrightarrow
U\times U^{\prime}$ is a closed immersion because it is the pullback of the
diagonal $\Delta_{V}\hookrightarrow V\times V$,
\[
\begin{tikzcd}
U\cap U^{\prime}\arrow{r}\arrow{d}&U\times U^{\prime}\arrow{d}\\
\Delta_V\arrow{r}&V\times V.
\end{tikzcd}
\]
Now $U\cap U^{\prime}$ is an affine scheme because it is a closed subscheme of
an affine scheme.

\begin{E}
\label{dt06}Let $V$ be an algebraic variety over $k$. If $k$ is separably
closed, then $V(k)$ is dense in $|V|$ (for the Zariski topology).
\end{E}

We may assume that $V$ is irreducible. Then $k(V)$ admits a separating
transcendence basis over $k$. This means that $V$ is birationally equivalent
to a hypersurface
\[
f(X_{1},\ldots,X_{d+1})=0,\quad d=\dim V,
\]
where $f$ has the property that $\partial f/\partial X_{d+1}\neq0$. This
implies that the closed points $P$ such that $k(P)$ is separable over $k$ form
a dense subset of $|V|$. In particular, $V(k)$ is dense in $|V|$ when $k$ is
separably closed.

\begin{E}
\label{dt02}Let $V$ be a quasi-projective scheme over an infinite field. Every
finite set of closed points of $V$ is contained in an open affine subset.
\end{E}

Embed $V$ as a subscheme of $\mathbb{P}{}^{n}$. Let $\bar{V}$ be the closure
of $V$ in $\mathbb{P}{}^{n}$, and let $Z=\bar{V}\smallsetminus V$ be the
boundary. For each $P\in S$, there exists a homogeneous polynomial $F_{P}\in
I(Z)$ such that $F_{P}(P)\neq0$. We may suppose that the $F_{P}$ have the same
degree. Because $k$ is infinite, some linear combination $F$ of the $F_{P}$
has the property that, for all $P\in S$, $F(P)\neq0$. Then $\bar{V}\cap D(F)$
is an open affine subset of $V$ containing $S$.

\begin{E}
\label{dt03} Let $A\subset B$ be rings with $B$ integral over $A$. Let
$\mathfrak{p}$ be a prime ideal of $A$. Then there exists a prime ideal
$\mathfrak{q}$ of $B$ such that $\mathfrak{p}{}=\mathfrak{q}{}\cap A$. If
$\mathfrak{q}{}^{\prime}\supset\mathfrak{q}{}$ is a second such prime ideal,
then $\mathfrak{q}^{\prime}=\mathfrak{q}{}$.
\end{E}

See, for example, 7.3 and 7.5 of my notes \textit{A Primer of Commutative
Algebra.}

\begin{E}
[Chevalley's theorem]\label{dt04}Let $\phi\colon W\rightarrow V$ be a dominant
morphism of irreducible algebraic schemes over $k$. Then $\phi(W)$ contains a
dense open subset of $V$.
\end{E}

See, for example, Theorem 15.8 of my notes \textit{A Primer of Commutative
Algebra.}

\begin{E}
\label{dt05}Let $A$ and $B$ be $k$-algebras. Assume that $k$ is algebraically
closed and $A$ is finitely generated over $k$.

\begin{enumerate}
\item If $A$ and $B$ are reduced. so also is $A\otimes_{k}B$.

\item If $A$ and $B$ are integral domains, so also is $A\otimes_{k}B$.
\end{enumerate}
\end{E}

Let $\alpha\in A\otimes_{k}B$. Then $\alpha=\sum_{i=1}^{n}a_{i}\otimes b_{i}$,
some $a_{i}\in A$, $b_{i}\in B$. If one of the $b_{j}$ is a $k$-linear
combination of the remaining $b_{i}$, say, $b_{n}=\sum_{i=1}^{n-1}c_{i}b_{i}$,
$c_{i}\in k$, then, using the bilinearity of $\otimes$, we find that
\[
\alpha=\sum_{i=1}^{n-1}a_{i}\otimes b_{i}+\sum_{i=1}^{n-1}c_{i}a_{n}\otimes
b_{i}=\sum_{i=1}^{n-1}(a_{i}+c_{i}a_{n})\otimes b_{i}.
\]
Thus we can suppose that in the original expression of $\alpha$, the $b_{i}$
are linearly independent over $k$.

Now assume $A$ and $B$ to be reduced, and suppose that $\alpha$ is nilpotent.
Let ${\mathfrak{m}}$ be a maximal ideal of $A$. From $a\mapsto\bar{a}\colon
A\rightarrow A/{\mathfrak{m}}=k$ we obtain homomorphisms
\[
a\otimes b\mapsto\bar{a}\otimes b\mapsto\bar{a}b\colon A\otimes_{k}%
B\rightarrow k\otimes_{k}B\overset{\simeq}{\rightarrow}B.
\]
The image $\sum\bar{a}_{i}b_{i}$ of $\alpha$ under this homomorphism is a
nilpotent element of $B$, and hence is zero (because $B$ is reduced). As the
$b_{i}$ are linearly independent over $k$, this means that the $\bar{a}_{i}$
are all zero. Thus, the $a_{i}$ lie in all maximal ideals ${\mathfrak{m}}$ of
$A$, and so are zero (because $A$ is reduced). Hence $\alpha=0$, and we have
shown that $A\otimes_{k}B$ is reduced.

Now assume that $A$ and $B$ are integral domains, and let $\alpha$,
$\alpha^{\prime}\in A\otimes_{k}B$ be such that $\alpha\alpha^{\prime}=0$. As
before, we can write $\alpha=\sum a_{i}\otimes b_{i}$ and $\alpha^{\prime
}=\sum a_{i}^{\prime}\otimes b_{i}^{\prime}$ with the sets $\{b_{1}%
,b_{2},\ldots\}$ and $\{b_{1}^{\prime},b_{2}^{\prime},\ldots\}$ each linearly
independent over $k$. For each maximal ideal ${\mathfrak{m}}$ of $A$, we know
that $(\sum\bar{a}_{i}b_{i})(\sum\bar{a}_{i}^{\prime}b_{i}^{\prime})=0$ in
$B$, and so either $(\sum\bar{a}_{i}b_{i})=0$ or $(\sum\bar{a}_{i}^{\prime
}b_{i}^{\prime})=0$. Thus either all the $a_{i}\in{\mathfrak{m}}$ or all the
$a_{i}^{\prime}\in{\mathfrak{m}}$. This shows that
\[
\spm(A)=V(a_{1},\ldots,a_{m})\cup V(a_{1}^{\prime},\ldots,a_{n}^{\prime}).
\]
As $\spm(A)$ is irreducible, it follows that $\spm(A)$ equals either
$V(a_{1},\ldots,a_{m})$ or $V(a_{1}^{\prime},\ldots,a_{n}^{\prime})$. In the
first case $\alpha=0$, and in the second $\alpha^{\prime}=0$.

\begin{remark}
\label{dt6a}If $K$ and $L$ are fields containing $k$, then $K\otimes_{k}L$
need by not be reduced, and if it is reduced, then it need not be an integral
domain. For example, let $K=k[\alpha]$, where $\alpha^{p}=a\in k$, but
$\alpha\notin k$. Then $K$ is a field, but $K\otimes_{k}K$ contains the
nilpotent element $\alpha\otimes1-1\otimes\alpha$. On the other hand, if $K$
is a separable extension of $k$ and $L$ is a Galois closure of $K/k$, then%
\[
K\otimes_{k}L\simeq\prod\nolimits_{\sigma\colon K\rightarrow L}L_{\sigma
}\text{,}%
\]
where $L_{\sigma}$ is a copy of $L$. Thus, \ref{dt05} may fail if $k$ is not
algebraically closed. However, if $A$ and $B$ are finitely generated reduced
$k$-algbras and $k$ is perfect, then $A\otimes_{k}B$ is reduced.
\end{remark}

\begin{notation}
\label{dt44} Let $V$ be an algebraic scheme over a field $F$. For a
homomorphism of fields $i\colon F\rightarrow L$, we sometimes write $iV$ for
\[
V_{L}\overset{\df}{=}V\times_{\Spec F}\Spec L.
\]
For example, if $V$ is embedded in affine space, then we get $iV$ by applying
$i$ to the coefficients of the polynomials defining $V$. If $\sigma
\in\Aut(L/iF)$, so $\sigma\circ i=i$, then $\left(  \sigma\circ i\right)
V\simeq iV$. We often view this as an equality $\sigma V_{L}=V_{L}$.

A morphism $\varphi\colon V\rightarrow W$ defines a morphism $\varphi
_{L}\colon V_{L}\rightarrow W_{L}$, which we sometimes denote $i\varphi\colon
iV\rightarrow iW$. Note that $(i\varphi)(iZ)=i(\varphi(Z))$ for any algebraic
subscheme $Z$ of $V$. For schemes embedded in affine space, $i\varphi$ is
obtained from $\varphi$ by applying $i$ to the coefficients of the polynomials
defining $\varphi$.
\end{notation}

\section{Models}

Let $\Omega\supset k$ be fields, and let $V$ be an algebraic scheme over
$\Omega$. A \emph{model} of $V$ over $k$ is an algebraic scheme $V_{0}$ over
$k$ together with an isomorphism $\varphi\colon V\rightarrow V_{0\Omega}$. An
algebraic scheme over $\Omega$ need not have a model over $k$, and when it
does it typically will have many nonisomorphic models.\footnote{For example,
an elliptic curve $E$ over $\mathbb{C}$ has a model over a number field if and
only if its $j$-invariant $j(E)$ is an algebraic number. If $Y^{2}Z=X^{3}
+aXZ^{2}+bZ^{3}$ is one model of $E$ over a number field $k$ (meaning, $a,b\in
k$), then $Y^{2}Z=X^{3}+ac^{2}XZ^{2}+bc^{3}Z^{3}$ is a second, which is
isomorphic to the first only if $c$ is a square in $k$.}

Let $V$ be an affine algebraic variety over $\Omega$. An embedding
$V\hookrightarrow\mathbb{A}_{\Omega}^{n}$ defines a model of $V$ over $k$ if
$I(V)$ is generated by polynomials in $k[X_{1},\ldots,X_{n}]$, because then
$I_{0}\overset{\df}{=}I(V)\cap k[X_{1},\ldots,X_{n}]$ is a radical ideal,
$k[X_{1},\ldots,X_{n}]/I_{0}$ is an affine $k$-algebra, and $V(I_{0}%
)\subset\mathbb{A}{}_{k}^{n}$ is a model of $V$. Moreover, every model of $V$
arises in this way from an embedding in affine space, because every model of
an affine algebraic variety is affine. However, different embeddings in affine
space will usually give rise to different models. Similar remarks apply to
projective varieties.

Note that the condition that $I(V)$ be generated by polynomials in
$k[X_{1},\ldots,X_{n}]$ is stronger than asking that $V$ be the zero set of
some polynomials in $k[X_{1},\ldots,X_{n}]$. For example, let $V=V(X+Y+\alpha
)$, where $\alpha$ is an element of $\Omega$ such that $\alpha^{p}\in k$ but
$\alpha\notin k$. Then $V$ is the zero set of the polynomial $X^{p}%
+Y^{p}+\alpha^{p}$, which has coefficients in $k$, but $I(V)=(X+Y+\alpha)$ is
not generated by polynomials in $k[X,Y]$.

\section{Fixed fields}

Let $\Omega\supset k$ be fields, and let $\Gamma$ be the group $\Aut(\Omega
/k)$ of automomorphisms of $\Omega$ (as an abstract field) fixing the elements
of $k$. Define the%
\index{field!fixed}
\emph{fixed field} $\Omega^{\Gamma}$ of $\Gamma$ to be%
\[
\{a\in\Omega\mid\sigma a=a\text{ for all }\sigma\in\Gamma\}.
\]

\begin{proposition}
\label{dt0}The fixed field of $\Gamma$ equals $k$ in each of the following two cases:

\begin{enumerate}
\item $\Omega$ is a Galois extension of $k$ (possibly infinite);

\item $\Omega$ is an algebraically closed field and $k$ is perfect.
\end{enumerate}
\end{proposition}

\begin{proof}
(a) See for example, \cite{milneFT}, 7.9.

(b) See for example, \cite{milneFT}, 9.29.
\end{proof}

\begin{remark}
\label{dt0m}(a) The proof of Proposition \ref{dt0} requires the axiom of
choice. For example, without the axiom of choice, every function $\mathbb{C}%
{}\rightarrow\mathbb{C}{}$ is measurable, hence continuous, but the only
continuous automorphisms of $\mathbb{C}{}$ are complex conjugation and the
identity map. Therefore, without the axiom of choice, $\mathbb{C}%
{}^{\Aut(\mathbb{C}{}/\mathbb{Q}{})}=\mathbb{R}{}$.

(b) Suppose that $\Omega$ is algebraically closed and $k$ is not perfect. Then
$k$ has characteristic $p\neq0$ and $\Omega$ contains an element $\alpha$ such
that $\alpha\notin k$ but $\alpha^{p}=a\in k$. As $\alpha$ is the unique root
of $X^{p}-a$, every automorphism of $\Omega$ fixing $k$ also fixes $\alpha$,
and so $\Omega^{\Gamma}\neq k$.
\end{remark}

The
\index{perfect closure}
\emph{perfect closure} of $k$ in $\Omega$ is the subfield
\[
k^{p^{-\infty}}=\{\alpha\in\Omega\mid\alpha^{p^{n}}\in k\text{ for some
}n\}\text{.}%
\]
The field $k^{p^{-\infty}}$ is purely inseparable over $k$. When $\Omega$ is
algebraically closed, it is the smallest perfect subfield of $\Omega$
containing $k$.

\begin{corollary}
\label{dt0c}If $\Omega$ is separably closed, then $\Omega^{\Aut(\Omega/k)}$ is
a purely inseparable algebraic extension of $k$. In particular, $\Omega
^{\Aut(\Omega/k)}=k$ if $k$ is perfect.
\end{corollary}

\begin{proof}
When $k$ has characteristic zero, $\Omega^{\Gamma}=k$, and there is nothing to
prove. Thus, we may suppose that $k$ has characteristic $p\neq0$. Choose an
algebraic closure $\Omega^{\mathrm{al}}$ of $\Omega$, and let $k^{p^{-\infty}%
}$ be the perfect closure of $k$ in $\Omega^{\mathrm{al}}$. As $\Omega
^{\mathrm{al}}$ is purely inseparable over $\Omega$, every element $\sigma$ of
$\Aut(\Omega/k)$ extends uniquely to an automorphism of $\Omega^{\mathrm{al}}%
$: if $\alpha\in\Omega^{\mathrm{al}}$, then $\alpha^{p^{n}}\in\Omega$ for some
$n$, and so an extension of $\sigma$ to $\Omega^{\mathrm{al}}$ must send
$\alpha$ to the unique root of $X^{p^{n}}-\sigma(\alpha^{p^{n}})$ in
$\Omega^{\mathrm{al}}$. The action of $\Aut(\Omega/k)$ on $\Omega
^{\mathrm{al}}$ identifies it with $\Aut(\Omega^{\mathrm{al}}/k^{p^{-\infty}%
})$. According to (b) of the proposition, $(\Omega^{\mathrm{al}})^{\Gamma
}=k^{p^{-\infty}}$, and so%
\[
k^{p^{-\infty}}\supset\Omega^{\Gamma}\supset k.
\]

\end{proof}

\section{Descending subspaces of vector spaces}

Let $\Omega\supset k$ be fields, and let $V$ be a $k$-subspace of an $\Omega
$-vector space $V(\Omega)$ such that the map
\begin{equation}
c\otimes v\mapsto cv\colon\Omega\otimes_{k}V\rightarrow V(\Omega) \label{e01}%
\end{equation}
is an isomorphism. This means that $\Omega V=V(\Omega)$ and that $k$-linearly
independent sets in $V$ are $\Omega$-linearly independent. Such $k$-spaces $V$
are the $k$-spans of $\Omega$-bases of $V(\Omega)$.

\begin{lemma}
\label{dt1a}Let $W$ be an $\Omega$-subspace of $V(\Omega)$. There exists at
most one $k$-subspace $W_{0}$ of $V$ such that $\Omega\otimes_{k}W_{0}$ maps
isomorphically onto $W$ under (\ref{e01}). The subspace $W_{0}$ exists if and
only if $V$ contains a set spanning $W$, in which case $W_{0}=V\cap W$.
\end{lemma}

\begin{proof}
If $W_{0}$ exists, then $\Omega W_{0}=W$, so it contains a set spanning $W$.
Conversely, if $V$ contains a set spanning $W$, then any $k$-basis for $V\cap
W$ is an $\Omega$-basis for $W$, and so $c\otimes w\mapsto cw\colon
\Omega\otimes_{k}(V\cap W)\rightarrow W$ is an isomorphism. No proper
$k$-subspace of $V\cap W$ can have this property.
\end{proof}

\begin{example}
\label{dt11b}Consider the fields $\mathbb{C}\supset\mathbb{Q}{}{}$, and let
$V=\mathbb{Q}{}^{2}$ and $V(\Omega)=\mathbb{\mathbb{C}{}}^{2}$. If $W$ is the
$\mathbb{C}{}$-subspace $\{(x,y)\in\mathbb{C}{}^{2}\mid y=\sqrt{2}x\}$ of
$V(\Omega)$, then $W\cap V=0$, and no $W_{0}$ exists.
\end{example}

Now assume that $k$ is the fixed field of $\Gamma\overset{\df}{=}%
\Aut(\Omega/k)$, and let $\Gamma$ act on $\Omega\otimes_{k{}}V$ through its
action on $\Omega$,%
\begin{equation}
\tstyle\sigma(\sum c_{i}\otimes v_{i})=\sum\sigma c_{i}\otimes v_{i}%
,\quad\sigma\in\Gamma,\quad c_{i}\in\Omega,\quad v_{i}\in V. \label{e22}%
\end{equation}
There is a unique action of $\Gamma$ on $V(\Omega)$ fixing the elements of $V$
and such that each $\sigma\in\Gamma$ acts $\sigma$-linearly,%
\begin{equation}
\sigma(cv)=\sigma(c)\sigma(v)\text{ all }\sigma\in\Gamma\text{, }c\in
\Omega\text{, }v\in V(\Omega{})\text{.} \label{e23}%
\end{equation}

\begin{lemma}
\label{dt2}We have $V=V(\Omega)^{\Gamma}$.
\end{lemma}

\begin{proof}
Let $(e_{i})_{i\in I}$ be a $k$-basis for $V$. Then $(1\otimes e_{i})_{i\in
I}$ is an $\Omega$-basis for $\Omega\otimes_{k}V$, and $\sigma\in\Gamma$ acts
on $v=\sum c_{i}\otimes e_{i}$ according to the rule (\ref{e22}). Thus, $v$ is
fixed by $\Gamma$ if and only if each $c_{i}$ is fixed by $\Gamma$ and so lies
in $k$.
\end{proof}

\begin{lemma}
\label{dt3}Let $W$ be a $\Omega$-subspace of $V(\Omega)$ stable under the
action of $\Gamma$. If $W\neq0$, then $W^{\Gamma}\neq0$.
\end{lemma}

\begin{proof}
As $V(\Omega)=\Omega V$, every nonzero element $w$ of $W$ can be expressed in
the form%
\[
w=c_{1}v_{1}+\cdots+c_{n}v_{n},\quad c_{i}\in\Omega\smallsetminus
\{0\},\quad\text{ }v_{i}\in V,\quad n\geq1.
\]
Let $w$ be a nonzero element of $W$ for which $n$ takes its smallest value.
After scaling, we may suppose that $c_{1}=1$. For $\sigma\in\Gamma$, the
element
\[
\sigma w-w=(\sigma c_{2}-c_{2})e_{2}+\cdots+(\sigma c_{n}-c_{n})e_{n}%
\]
lies in $W$ and has at most $n-1$ nonzero coefficients, and so is zero. Thus,
$w\in W^{\Gamma}$.
\end{proof}

\begin{proposition}
\label{dt3a}A subspace $W$ of $V(\Omega)$ is of the form $W=\Omega W_{0}$ for
some $k$-subspace $W_{0}$ of $V$ if and only if it is stable under the action
of $\Gamma$, in which case $W_{0}=V\cap W=W^{\Gamma}$.
\end{proposition}

\begin{proof}
Certainly, if $W=\Omega W_{0}$, then it is stable under $\Gamma$ (and
$W=\Omega(W\cap V)$). Conversely, assume that $W$ is stable under $\Gamma$,
and let $W^{\prime}$ be a complement to $W\cap V$ in $V$, so that%
\[
V=(W\cap V)\oplus W^{\prime}\text{.}%
\]
Then%
\[
(W\cap\Omega W^{\prime})^{\Gamma}=W^{\Gamma}\cap\left(  \Omega W^{\prime
}\right)  ^{\Gamma}=(W\cap V)\cap W^{\prime}=0\text{,}%
\]
and so, by Lemma \ref{dt3},
\begin{equation}
W\cap\Omega W^{\prime}=0\text{.} \label{e29}%
\end{equation}
As $W\supset\Omega(W\cap V)$ and
\[
V(\Omega)=\Omega(W\cap V)\oplus\Omega W^{\prime}\text{,}%
\]
this implies that $W=\Omega(W\cap V)$: write an element $w$ of $W$ as
$w=w_{1}+w_{2}$ with $w_{1}\in\Omega(W\cap V)$ and $w_{2}\in\Omega W^{\prime}%
$; then $w_{2}=w-w_{1}\in W\cap\Omega W^{\prime}$, and so it is $0$.
\end{proof}

\section{Descending subschemes of algebraic schemes}

Let $\Omega\supset k$ be fields.

\begin{proposition}
\label{dt4a}Let $V$ be an algebraic scheme over $k$, and let $W$ be a closed
subscheme of $V_{\Omega}$. There exists at most one closed subscheme $W_{0}$
of $V$ such that $W_{0\Omega}=W$ (as a subscheme of $V_{\Omega}$) .
\end{proposition}

\begin{proof}
If $V=\Spec A$ and $I$ is an ideal in $\Omega\otimes_{k}A$, then there is at
most one ideal $I_{0}$ in $A$ such that $\Omega\otimes_{k}I_{0}$ maps
isomorphically onto $I$ under $c\otimes a\mapsto ca$. Moreover, the ideal
$I_{0}$ exists if and only if $A$ contains a set of generators for the ideal
$I$, in which case $I_{0}=I\cap A$ (see~\ref{dt1a}). To prove the general
case, cover $V$ with open affines.
\end{proof}

\begin{proposition}
\label{dt4b}Let $V$ and $W$ be algebraic schemes over $k$, and let
$\varphi\colon V_{\Omega}\rightarrow W_{\Omega}$ be a morphism over $\Omega$.
There exists at most one morphism $\varphi_{0}\colon V\rightarrow W$ such that
$\varphi_{0\Omega}=\varphi$.
\end{proposition}

\begin{proof}
As $W$ is separated, the graph $\Gamma_{\varphi}$ of $\varphi$ is closed in
$V\times W$, and so we can apply \ref{dt4a}.
\end{proof}

Now assume that $k$ is perfect and $\Omega$ is separably closed. Then $k$ is
the fixed field of $\Gamma=\Aut(\Omega/k)$.

For any algebraic variety $V$ over $\Omega$, $V(\Omega)$ is Zariski dense in
$V$ (see \ref{dt06}). It follows that two regular maps $V\rightrightarrows W$
of algebraic varieties coincide if they agree on $V(\Omega)$.

For any algebraic scheme $V$ over $k$, $\Gamma$ acts on $V(\Omega)$. For
example, if $V$ is embedded in $\mathbb{A}{}^{n}$ or $\mathbb{P}{}^{n}$ over
$k$, then $\Gamma$ acts on the coordinates of a point. If $V=\Spec A$, then%
\[
V(\Omega)=\Hom(A,\Omega)\qquad(k\text{-algebra homomorphisms),}%
\]
and $\Gamma$ acts through its action on $\Omega$.

\begin{proposition}
\label{dt4}Let $V$ be an algebraic scheme over $k{}{}{}$, and let $W$ be a
reduced closed subscheme of $V_{\Omega}$. There exists a closed subscheme
$W_{0}$ of $V$ such that $W=W_{0\Omega}$ if and only if $W(\Omega)$ is stable
under the action of $\Gamma{}$ on $V(\Omega)$.
\end{proposition}

\begin{proof}
Certainly, the condition is necessary. For the converse, suppose first that
$V$ is affine, and let $I(W)$ be the ideal in $\Omega\lbrack V_{\Omega}]$
corresponding to $W$. Note that $\Omega\lbrack V_{\Omega}]=\Omega\otimes
_{k}k[V]$. Because $W(\Omega)$ is stable under $\Gamma$, so also is $I(W)$,
and Proposition \ref{dt3a} shows that $I(W)$ is spanned by $I_{0}%
\overset{\df}{=}I(W)\cap k{}[V]$. The closed subscheme $W_{0}$ of $V$
corresponding to $I_{0}$ has the property that $W=W_{0\Omega}$.

To deduce the general case, cover $V$ with open affines $V=\bigcup V_{i}$.
Then $W_{i}\overset{\df}{=}V_{i\Omega}\cap W$ is such that $W_{i}(\Gamma)$ is
stable under $\Gamma$, and so it arises from a closed subscheme $W_{i0}$ of
$V_{i}$; a similar statement holds for $W_{ij}\overset{\df}{=}W_{i}\cap W_{j}%
$. Define $W_{0}$ to be the scheme obtained by patching the $W_{i0}$ along the
open subschemes $W_{ij0}$.
\end{proof}

\begin{corollary}
\label{dt5}Let $V$ and $W$ be algebraic varieties over $k{}$, and let $f\colon
V_{\Omega}\rightarrow W_{\Omega}$ be a regular map. If $f(\Omega)\colon
V(\Omega)\rightarrow W(\Omega)$ commutes with the actions of $\Gamma$, then
$f$ arises from a (unique) regular map $V\rightarrow W$ over $k{}$.
\end{corollary}

\begin{proof}
Apply Proposition \ref{dt4} to the graph of $f$, $\Gamma_{f}\subset(V\times
W)_{\Omega}$.
\end{proof}

\begin{corollary}
\label{dt6}The functor%
\begin{equation}
V\rightsquigarrow(V_{\Omega},\mathrm{\,\,action\,\,of\,\,}\Gamma
\mathrm{\,\,on\,\,}V(\Omega)) \label{e31}%
\end{equation}
from algebraic varieties over $k$ to algebraic varieties over $\Omega$
equipped with an action of $\Gamma$ on their $\Omega$-points is fully faithful.
\end{corollary}

\begin{proof}
Restatement of \ref{dt5}.
\end{proof}

In particular, an algebraic variety $V$ over $k$ is uniquely determined up to
a unique isomorphism by the algebraic variety $V_{\Omega}$ equipped with the
action of $\Gamma$ on $V(\Omega)$.

In Theorems \ref{dt31} and \ref{dt30} below, we obtain sufficient conditions
for a pair to lie in the essential image of the functor (\ref{e31}).

\section{Galois descent of vector spaces}

Let $\Gamma$ be a group acting on a field $\Omega$, and let $k$ be a subfield
of $\Omega^{\Gamma}$. By a \emph{semilinear action} of $\Gamma$ on an $\Omega
$-vector space $V$ we mean a homomorphism $\rho\colon\Gamma\rightarrow
\Aut_{k\text{-linear}}(V)$ such that, for all $\sigma\in\Gamma$, $\rho
(\sigma)$ acts $\sigma$-linearly on $V$,%
\[
\rho(\sigma)(cv)=\sigma(c)v,\quad c\in\Omega,\quad v\in V.
\]
For example, if $V$ is a $k$-vector space, then $\sigma(c\otimes v)=\sigma
c\otimes v$ is a semilinear action of $\Gamma$ on $\Omega\otimes_{k}V$.

\begin{lemma}
\label{dt7}Let $S$ be the standard $M_{n}(k)$-module (i.e., $S=k^{n}$ with
$M_{n}(k)$ acting by left multiplication). The functor $V\rightsquigarrow
S\otimes_{k}V$ from $k$-vector spaces to left $M_{n}(k)$-modules is an
equivalence of categories.
\end{lemma}

\begin{proof}
Let $V$ and $W$ be $k$-vector spaces. The choice of bases $(e_{i})_{i\in I}$
and $(f_{j})_{j\in J}$ for $V$ and $W$ identifies $\Hom_{k}(V,W)$ with the set
of matrices $(a_{ji})_{(j,i)\in J\times I}$, $a_{ji}\in k$, such that, for a
fixed $i$, all but finitely many $a_{ji}$ are zero. Because $S$ is a simple
$M_{n}(k)$-module and $\End_{M_{n}(k)}(S)=k$, the set $\Hom_{M_{n}%
(k)}(S\otimes_{k}V,S\otimes_{k}W)$ has the same description, and so the
functor $V\rightsquigarrow S\otimes_{k}V$ from $k$-modules to left $M_{n}%
(k)$-modules is fully faithful.

The functor $V\rightsquigarrow S\otimes_{k}V$ sends a vector space $V$ with
basis $(e_{i})_{i\in I}$ to a direct sum of copies of $S$ indexed by $I$.
Therefore, to show that the functor is essentially surjective, it suffices to
prove that every left $M_{n}(k)$-module is a direct sum of copies of $S$.

We first prove this for $M_{n}(k)$ regarded as a left $M_{n}(k)$-module. For
$1\leq i\leq n$, let $L(i)$ be the set of matrices in $M_{n}(k)$ whose entries
are zero except for those in the $i$th column. Then $L(i)$ is a left ideal in
$M_{n}(k)$, and $L(i)$ is isomorphic to $S$ as an $M_{n}(k)$-module. Hence,
\[
M_{n}(k)=\bigoplus_{i}L(i)\simeq S^{n}\qquad\text{(as a left }M_{n}%
(k)\text{-module).}%
\]

We now prove it for an arbitrary (nonzero) left $M_{n}(k)$-module $M$. The
choice of a set of generators for $M$ realizes it as a quotient of a sum of
copies of $M_{n}(k)$, and so $M$ is a sum of copies of $S$. It remains to show
that the sum can be made direct. Let $I$ be the set of submodules of $M$
isomorphic to $S$, and let $\Xi$ be the set of subsets $J$ of $I$ such that
the sum $N(J)\overset{\df}{=}\sum_{N\in J}N$ is direct, i.e., such that for
any $N_{0}\in J$ and finite subset $J_{0}$ of $J$ not containing $N_{0}$,
$N_{0}\cap\sum_{N\in J_{0}}N=0$. If $J_{1}\subset J_{2}\subset\ldots$ is a
chain of sets in $\Xi$, then $\bigcup J_{i}\in\Xi$, and so Zorn's lemma
implies that $\Xi$ has maximal elements. For any maximal $J$, $M=N(J)$ because
otherwise, there exists an element $S^{\prime}$ of $I$ not contained in
$N(J)$; because $S^{\prime}$ is simple, $S^{\prime}\cap N(J)=0$, and it
follows that $J\cup\{S^{\prime}\}\in\Xi$, contradicting the maximality of $J$.
\end{proof}

\begin{aside}
\label{dt7n}The above argument proves the following statement: let $A$ be a
ring (not necessarily commutative) and $S$ a simple left $A$-module; if
$_{A}A$ is a sum of submodules isomorphic to $S$, then every left $A$-module
is a direct sum of submodules isomorphic to $S$.
\end{aside}

\begin{aside}
\label{dt7m}Let $A$ and $B$ be rings (not necessarily commutative), and let
$S$ be $A$-$B$-bimodule (this means that $A$ acts on $S$ on the left, $B$ acts
on $S$ on the right, and the actions commute). When the functor
$M\rightsquigarrow S\otimes_{B}M\colon\Mod_{B}\rightarrow\Mod_{A}$ is an
equivalence of categories, $A$ and $B$ are said to be%
\index{Morita equivalent}
\emph{Morita equivalent through }$S$. In this terminology, the lemma says that
$M_{n}(k)$ and $k$ are Morita equivalent through $S$.
\end{aside}

\begin{proposition}
\label{dt8}Let $\Omega$ be a finite Galois extension of $k$ with Galois group
$\Gamma$. The functor $V\rightsquigarrow(\Omega\otimes_{k}V,\ast)$ from
$k$-vector spaces to $\Omega$-vector spaces endowed with a semilinear action
of $\Gamma$ is an equivalence of categories.
\end{proposition}

\begin{proof}
Let $\Omega\lbrack\Gamma]~$be the $\Omega$-vector space with basis
$\{\sigma\in\Gamma\}$, and make $\Omega\lbrack\Gamma]$ into a $k$-algebra by
setting%
\[
\tstyle\big(\sum\limits_{\sigma\in\Gamma}a_{\sigma}\sigma\big)\big(\sum
\limits_{\tau\in\Gamma}b_{\tau}\tau\big)=\sum\limits_{\sigma,\tau}(a_{\sigma
}\cdot\sigma b_{\tau})\sigma\tau\text{.}%
\]
Then $\Omega\lbrack\Gamma]$ acts $k$-linearly on $\Omega$ by the rule%
\[
\tstyle(\sum_{\sigma\in\Gamma}a_{\sigma}\sigma)c =\sum_{\sigma\in\Gamma
}a_{\sigma}(\sigma c),
\]
and Dedekind's theorem on the independence of characters (\cite{milneFT},
5.14) implies that the homomorphism%
\[
\Omega\lbrack\Gamma]\rightarrow\End_{k}(\Omega)
\]
defined by this action is injective. By counting dimensions over $k$, one sees
that it is an isomorphism. Therefore, Lemma \ref{dt7} shows that
$\Omega\lbrack\Gamma]$ and $k$ are Morita equivalent through $\Omega$, i.e.,
the functor $V\mapsto\Omega\otimes_{k}V$ from $k$-vector spaces to left
$\Omega\lbrack\Gamma]$-modules is an equivalence of categories. This is
precisely the statement of the lemma.
\end{proof}

When $\Omega$ is an infinite Galois extension of $k$, we endow $\Gamma$ with
the Krull topology, and we say that a semilinear action of $\Gamma$ on an
$\Omega$-vector space $V$ is%
\index{action!continuous}
\emph{continuous }if every element of $V$ is fixed by an open subgroup of
$\Gamma$, i.e., if%
\[
V=\bigcup\nolimits_{\Delta}V^{\Delta}\qquad(\text{union over the open
subgroups }\Delta\text{ of }\Gamma\text{).}%
\]
For example, the action of $\Gamma$ on $\Omega$ is continuous, and it follows
that, for any $k$-vector space $V$, the action of $\Gamma$ on $\Omega
\otimes_{k}V$ is continuous.

\begin{proposition}
\label{dt10}Let $\Omega$ be a Galois extension of $k$ (possibly infinite) with
Galois group $\Gamma$. For any $\Omega$-vector space $V$ equipped with a
continuous semilinear action of $\Gamma$, the map%
\[
\tstyle\sum c_{i}\otimes v_{i}\mapsto\sum c_{i}v_{i}\colon{}\Omega\otimes
_{k}V^{\Gamma}\rightarrow V
\]
is an isomorphism.
\end{proposition}

\begin{proof}
Suppose first that $\Gamma$ is finite. According to Proposition \ref{dt8},
there is a subspace $W$ of $V$ such that $\Omega\otimes_{k}W\simeq V$.
Moreover, $W=V^{\Gamma}$ by \ref{dt2}, and so $\Omega\otimes_{k}V^{\Gamma
}\simeq V$.

When $\Gamma$ is infinite, the finite case shows that $\Omega\otimes
_{k}(V^{\Delta})^{\Gamma/\Delta}\simeq V^{\Delta}$ for every open normal
subgroup $\Delta$ of $\Gamma$. Now pass to the direct limit over $\Delta$,
recalling that tensor products commute with direct limits.
\end{proof}

\begin{proposition}
\label{dt10d}The functor%
\[
W\rightsquigarrow(\Omega\otimes_{k}W,\ast)
\]
from $k$-vector spaces to $\Omega$-vector spaces equipped with a continuous
semilinear action of $\Gamma$ is an equivalence of categories, with
quasi-inverse $(V,\ast)\rightsquigarrow V^{\Gamma}.$
\end{proposition}

\begin{proof}
We have constructed natural isomorphisms $W\simeq(\Omega\otimes_{k}W)^{\Gamma
}$ (see \ref{dt2}) and $\Omega\otimes_{k}V^{\Gamma}\simeq V$ (see \ref{dt10}).
\end{proof}

\section{Descent data}

Let $\Omega\supset k$ be fields, and let $\Gamma=\Aut(\Omega/k)$. An
$\Omega/k$-\emph{descent system }%
\index{descent system}
on an algebraic scheme $V$ over $\Omega$ is a family $(\varphi_{\sigma
})_{\sigma\in\Gamma}$ of isomorphisms $\varphi_{\sigma}\colon\sigma
V\rightarrow V$ satisfying the cocycle condition,%
\[
\varphi_{\sigma}\circ(\sigma\varphi_{\tau})=\varphi_{\sigma\tau}\text{ for all
}\sigma,\tau\in\Gamma,
\]%
\[
\begin{tikzcd}
\sigma\tau V\arrow[bend left=20]{rr}{\varphi_{\sigma\tau}}
\arrow{r}[swap]{\sigma\varphi_{\tau}}
&\sigma V\arrow{r}[swap]{\varphi_{\sigma}}
&V.
\end{tikzcd}
\]
A model $(V_{0},\varphi)$ of $V$ over a subfield $K$ of $\Omega$ containing
$k$%
\index{splits!a descent system}
\emph{splits} the descent system if $\varphi_{\sigma}=\varphi^{-1}\circ
\sigma\varphi$ for all $\sigma$ fixing $K$,
\[
\begin{tikzcd}
\sigma V\arrow[bend left=20]{rr}{\varphi_{\sigma}}\arrow{r}[swap]{\sigma\varphi}
&\sigma(V_{0\Omega})=V_{0\Omega}&V\arrow{l}{\varphi}
\end{tikzcd}
\]

A descent system $(\varphi_{\sigma})_{\sigma\in\Gamma}$ is%
\index{continuous!descent system}
said to be \emph{continuous }if it is split by some model over a subfield $K$
of $\Omega$ that is \textit{finitely generated} over $k$. A%
\index{descent datum}
\emph{descent datum }is a continuous descent system. A descent datum is%
\index{descent datum!effective}
\emph{effective }if it is split by some model over $k$. In a given situation,
we say that \emph{descent is effective }if every descent datum is effective.

Let $V_{0}$ be an algebraic scheme over $k$, and let $V=V_{0\Omega}$. For
$\sigma\in\Gamma$, let $\varphi_{\sigma}$ denote the canonical isomorphism
$\sigma V\rightarrow V$. Then $(\varphi_{\sigma})_{\sigma\in\Gamma}$ is an
$\Omega/k$-descent datum, split by $V_{0}$.

Let $(\varphi_{\sigma})_{\sigma\in\Aut(\Omega/k)}$ be an $\Omega/k$-descent
system on an algebraic scheme $V$ over $\Omega$, and let $\Omega
^{\mathrm{sep}}$ be a separable closure of $\Omega$. The restriction map
$\Aut(\Omega^{\mathrm{sep}}/k)\rightarrow\Aut(\Omega/k)$ is surjective, and we
can extend$(\varphi_{\sigma})_{\sigma\in\Aut(\Omega/k)}$ to an $\Omega
^{\mathrm{sep}}/k$-descent system $(\varphi_{\sigma}^{\prime})_{\sigma
\in\Aut(\Omega^{\mathrm{sep}}/k)}$ on $V_{\Omega^{\mathrm{sep}}}$ by setting
$\varphi_{\sigma}^{\prime}=(\varphi_{\sigma|\Omega})_{\Omega^{\mathrm{sep}}}$.
A model of $V$ over a subfield $K$ of $\Omega$ splits $(\varphi_{\sigma
})_{\sigma}$ if and only if it splits $(\varphi_{\sigma}^{\prime})_{\sigma}$.
This observation sometimes allows us to assume that $\Omega$ is separably closed.

\begin{proposition}
\label{dt10a}Assume that $k=\Omega^{\Aut(\Omega/k)}$, and let $(\varphi
_{\sigma})_{\sigma\in\Gamma}$ and $(\varphi_{\sigma}^{\prime})_{\sigma
\in\Gamma}$ be $\Omega/k$-descent data on algebraic varieties $V$ and
$V^{\prime}$ over $\Omega.$ If $(V_{0},\varphi)$ and $(V_{0}^{\prime}%
,\varphi^{\prime})$ are models over $k$ splitting $(\varphi_{\sigma}%
)_{\sigma\in\Gamma}$ and $(\varphi_{\sigma}^{\prime})_{\sigma\in\Gamma}$
respectively, then to give a regular map $\alpha_{0}\colon V_{0}\rightarrow
V_{0}$ is the same as giving a regular map $\alpha\colon V\rightarrow
V^{\prime}$ such that diagrams%
\begin{equation}
\begin{tikzcd} \sigma V\arrow{r}{\varphi_{\sigma}}\arrow{d}{\sigma\alpha} &V\arrow{d}{\alpha}\\ \sigma V^{\prime}\arrow{r}{\varphi_{\sigma}^{\prime}} &V^{\prime}\end{tikzcd} \label{e32}%
\end{equation}
commute for all $\sigma\in\Gamma$.
\end{proposition}

\begin{proof}
Given $\alpha_{_{0}}$, define $\alpha$ to make the right hand square in%
\[
\begin{tikzcd}
\sigma V \arrow{r}{\sigma\varphi}\arrow{d}{\sigma\alpha}
&V_{0\Omega}\arrow{d}{\alpha_{0\Omega}}
&V\arrow{l}[swap]{\varphi}\arrow{d}{\alpha}\\
\sigma V^{\prime}\arrow{r}{\sigma\varphi^{\prime}}
&V_{0\Omega}^{\prime}\arrow{r}{\varphi^{\prime}}&V^{\prime}%
\end{tikzcd}
\]
commute. The left hand square is obtained from the right hand square by
applying $\sigma$, and so it also commutes. The outer square is (\ref{e32}).

In proving the converse, we may assume that $\Omega$ is separably closed.
Given $\alpha$, use $\varphi$ and $\varphi^{\prime}$ to transfer $\alpha$ to a
regular map $\alpha^{\prime}\colon V_{0\Omega}\rightarrow V_{0\Omega}^{\prime
}$. Then the hypothesis implies that $\alpha^{\prime}$ commutes with the
actions of $\Gamma$ on $V_{0}(\Omega)$ and $V_{0}^{\prime}(\Omega)$, and so is
defined over $k$ (\ref{dt5}).
\end{proof}

\begin{corollary}
\label{dt10c}Assume that $k=\Omega^{\Aut(\Omega/k)}$. Let $(\varphi_{\sigma
})_{\sigma\in\Gamma}$ be a descent datum on a variety $V$ over $\Omega$, and
let $(V_{0},\varphi)$ be a model over $k$ splitting $(\varphi_{\sigma
})_{\sigma\in\Gamma}$. Let $W$ be an algebraic scheme over $k$. To give a
regular map $W\rightarrow V_{0}$ (resp. $V_{0}\rightarrow W$) is the same as
giving a regular map $\alpha\colon W_{\Omega}\rightarrow V$ (resp.
$\alpha\colon V\rightarrow W_{\Omega}$) compatible with the descent datum,
i.e., such that $\varphi_{\sigma}\circ\sigma\alpha=\alpha$ (resp. $\alpha
\circ\varphi_{\sigma}=\sigma\alpha$).
\end{corollary}

\begin{proof}
This is the special case of the proposition in which $W_{\Omega}$ is endowed
with its canonical descent datum.
\end{proof}

\begin{remark}
\label{dt10b}Proposition \ref{dt10a} shows that the functor taking an
algebraic variety $V$ over $k$ to $V_{\Omega}$ endowed with its canonical
descent datum,%
\[
\{\text{varieties over }k\}\rightarrow\{\text{varieties over }\Omega
+\Omega/k\text{-descent datum}\}
\]
is fully faithful. We are interested in determining when it is essentially surjective.
\end{remark}

Let $(\varphi_{\sigma})_{\sigma\in\Gamma}$ be an $\Omega/k$-descent system on
$V$. For a subscheme $W$ of $V$, we set
\[
{}^{\sigma}W=\varphi_{\sigma}(\sigma W).
\]
Then the following diagram commutes,%
\begin{equation}
\begin{tikzcd} \sigma V \arrow{r}{\varphi_{\sigma}}[swap]{\simeq} &V\\ \sigma W\arrow[hook]{u}\arrow{r}{\varphi_{\sigma}|\sigma W}[swap]{\simeq} &{}^{\sigma}W.\arrow[hook]{u} \end{tikzcd} \label{e02}%
\end{equation}

\begin{lemma}
\label{dt11}The following hold.

\begin{enumerate}
\item For all $\sigma,\tau\in\Gamma$ and $W\subset V$, ${}^{\sigma}({}^{\tau
}W)={}^{\sigma\tau}W$.

\item Suppose that $(\varphi_{\sigma})_{\sigma\in\Gamma}$ is split by a model
$(V_{0},\varphi)$ of $V$ over $k_{0}$, and let $W$ be an algebraic subscheme
of $V$. If $W=\varphi^{-1}(W_{0\Omega})$ for some algebraic subscheme $W_{0}$
of $V_{0}$, then ${}^{\sigma}W=W$ for all $\sigma\in\Gamma$; the converse is
true if $\Omega^{\Gamma}=k$.
\end{enumerate}
\end{lemma}

\begin{proof}
(a) By definition%
\[
^{\sigma}({}^{\tau}W)=\varphi_{\sigma}(\sigma(\varphi_{\tau}(\tau
W))=(\varphi_{\sigma}\circ\sigma\varphi_{\tau})(\sigma\tau W)=\varphi
_{\sigma\tau}(\sigma\tau W)={}^{\sigma\tau}W\text{.}%
\]
In the second equality, we used that $(\sigma\varphi)(\sigma W)=\sigma(\varphi
W)$.

(b) Let $W=\varphi^{-1}(W_{0\Omega})$. By hypothesis $\varphi_{\sigma}%
=\varphi^{-1}\circ\sigma\varphi$, and so%
\[
^{\sigma}W=(\varphi^{-1}\circ\sigma\varphi)(\sigma W)=\varphi^{-1}%
(\sigma(\varphi W))=\varphi^{-1}(\sigma W_{0\Omega})=\varphi^{-1}(W_{0\Omega
})=W.
\]
Conversely, suppose $^{\sigma}W=W$ for all $\sigma\in\Gamma$. Then%
\[
\varphi(W)=\varphi(^{\sigma}W)=(\sigma\varphi)(\sigma W)=\sigma(\varphi
(W))\text{.}%
\]
Therefore, $\varphi(W)$ is stable under the action of $\Gamma$ on $V_{0\Omega
}$, and so is defined over $k$ (see \ref{dt4}).
\end{proof}

For a descent system $(\varphi_{\sigma})_{\sigma\in\Gamma}$ on $V$ and a
regular function $f$ on an open subset $U$ of $V$, we define $^{\sigma}f$ to
be the function $(\sigma f)\circ\varphi_{\sigma}^{-1}$ on ${}^{\sigma}U$, so
that $^{\sigma}f({}^{\sigma}P)=\sigma\left(  f(P)\right)  $ for all $P\in U$.
Then ${}^{\sigma}({}^{\tau}f)={}^{\sigma\tau}f$, and so this defines an action
of $\Gamma$ on the regular functions.

The
\index{topology!Krull}
\emph{Krull topology} on $\Gamma$ is that for which the subgroups of $\Gamma$
fixing a subfield of $\Omega$ finitely generated over $k$ form a basis of open
neighbourhoods of $1$ (see, for example, \cite{milneFT}, Chapter 7). An action
of $\Gamma$ on an $\Omega$-vector space $V$ is \emph{continuous} if%
\[
V=\bigcup_{\Delta}V^{\Delta}\qquad(\text{union over the open subgroups }%
\Delta\text{ of }\Gamma\text{).}%
\]
For a subfield $L$ of $\Omega$ containing $k$, let $\Delta_{L}=\Aut(\Omega/L)$.

\begin{proposition}
\label{dt12}Assume that $\Omega$ is separably closed. An $\Omega/k$-descent
system $(\varphi_{\sigma})_{\sigma\in\Gamma}$ on an affine algebraic scheme
$V$ is continuous if and only if the action of $\Gamma$ on $\Omega\lbrack V]$
is continuous.
\end{proposition}

\begin{proof}
If $(\varphi_{\sigma})_{\sigma\in\Gamma}$ is continuous, it is split by a
model of $V$ over some subfield $K$ of $\Omega$ finitely generated over $k$.
By definition, $\Delta_{K}$ is open, and $\Omega\lbrack V]^{\Delta_{K}}$
contains a set $\{f_{1},\ldots,f_{n}\}$ of generators for $\Omega\lbrack V]$
as an $\Omega$-algebra. Now $\Omega\lbrack V]=\bigcup L[f_{1},\ldots,f_{n}]$,
where $L$ runs over the subfields of $\Omega$ containing $K$ and finitely
generated over $k$. As $L[f_{1},\ldots,f_{n}]\subset\Omega\lbrack
V]^{\Delta_{L}}$, this shows that $\Omega\lbrack V]=\bigcup\Omega\lbrack
V]^{\Delta_{L}}$.

Conversely, if the action of $\Gamma$ on $\Omega\lbrack V]$ is continuous,
then for some subfield $L$ of $\Omega$ finitely generated over $k$,
$\Omega\lbrack V]^{\Delta_{L}}$ will contain a set of generators $f_{1}%
,\ldots,f_{n}$ for $\Omega\lbrack V]$ as an $\Omega$-algebra. According to
\ref{dt0c}, $\Omega^{\Delta_{L}}$ is a purely inseparable algebraic extension
of $L$, and so, after possibly replacing $L$ with a finite extension, we may
suppose that the embedding $V\hookrightarrow\mathbb{A}{}^{n}$ defined by the
$f_{i}$ determine a model of $V$ over $L$. This model splits $(\varphi
_{\sigma})_{\sigma\in\Gamma}$, which is therefore continuous.
\end{proof}

\begin{proposition}
\label{dt13}A descent system $(\varphi_{\sigma})_{\sigma\in\Gamma}$ on an
algebraic scheme $V$ over $\Omega$ is continuous if there exists a finite set
$S$ of points in $V(\Omega)$ such that

\begin{enumerate}
\item the only automorphism of $V$ fixing all $P\in S$ is the identity, and

\item there exists a subfield $K$ of $\Omega$ finitely generated over $k$ such
that ${}^{\sigma}P=P$ for all $\sigma\in\Gamma$ fixing $K$.
\end{enumerate}
\end{proposition}

\begin{proof}
Let $(V_{0},\varphi)$ be a model of $V$ over a subfield $K$ of $\Omega$
finitely generated over $k$. After possibly replacing $K$ by a larger finitely
generated field, we may suppose (i) that ${}^{\sigma}P=P$ for all $\sigma
\in\Gamma$ fixing $K$ and all $P\in S$ (because of (b)) and (ii) that
$\varphi(P)\in V_{0}(K)$ for all $P\in S$ (because $S$ is finite). Then, for
$P\in S$ and every $\sigma$ fixing $K$,%
\begin{align*}
&  \varphi_{\sigma}(\sigma P)\overset{\df}{=}\,^{\sigma}P\overset{\text{(i)}%
}{=}P\\
&  (\sigma\varphi)(\sigma P)=\sigma(\varphi P)\overset{\text{(ii)}}{=}\varphi
P,
\end{align*}
and so both $\varphi_{\sigma}$ and $\varphi^{-1}\circ\sigma\varphi$ are
isomorphisms $\sigma V\rightarrow V$ sending $\sigma P$ to $P$. Therefore,
$\varphi_{\sigma}$ and $\varphi^{-1}\circ\sigma\varphi$ differ by an
automorphism of $V$ fixing the $P\in S$, and so are equal. This says that
$(V_{0},\varphi)$ splits $(\varphi_{\sigma})_{\sigma\in\Gamma}$.
\end{proof}

\begin{proposition}
\label{dt13m}Let $V$ be an algebraic scheme over $\Omega$ such that the only
automorphism of $V$ is the identity map. If $V$ has a model over $k$, then
every $\Omega/k$-descent datum on $V$ is effective. More precisely, every
$\Omega/k$-descent datum on $V$ is split by the model.
\end{proposition}

\begin{proof}
If $(\varphi_{\sigma})_{\sigma}$ is a descent datum on $V$ and $(V,\varphi)$
is a model of $V$ over $k$, then $\varphi_{\sigma}$ and $\varphi^{-1}%
\circ\sigma\varphi$ are both isomorphisms $\sigma V\rightarrow V$, hence
differ by an automorphism of $V$. Thus $\varphi_{\sigma}=\varphi^{-1}%
\circ\sigma\varphi$.
\end{proof}

Of course, in Proposition \ref{dt13}, $S$ does not have to be a finite set of
points. The proposition will hold with $S$ any additional structure on $V$
that rigidifies $V$ (i.e., such that $\Aut(V,S)=1$) and is such that $(V,S)$
has a model over a finitely generated extension of $k$.

\section{Galois descent of algebraic schemes}

In this section, $\Omega$ is a Galois extension of $k$ with Galois group
$\Gamma$.

\begin{theorem}
\label{dt14}A descent datum $(\varphi_{\sigma})_{\sigma\in\Gamma}$ on an
algebraic scheme $V$ is effective if $V$ is covered by open affines $U$ with
the property that ${}^{\sigma}U=U$ for all $\sigma\in\Gamma$.
\end{theorem}

\begin{proof}
Assume first that $V$ is affine, and let $A=k[V]$. A descent datum
$(\varphi_{\sigma})_{\sigma\in\Gamma}$ defines a continuous action of $\Gamma$
on $A$ (see \ref{dt12}). From \ref{dt10}, we know that the map
\begin{equation}
c\otimes a\mapsto ca\colon\Omega\otimes_{k}A^{\Gamma}\rightarrow A \label{e27}%
\end{equation}
is an isomorphism. Let $V_{0}=\Spec A^{\Gamma}$, and let $\varphi$ be the
isomorphism $V\rightarrow V_{0\Omega}$ defined by (\ref{e27}). Then
$(V_{0},\varphi)$ splits the descent datum.

Next note that if $^{\sigma}U=U$ for all $\sigma\in\Gamma$, then a descent
datum on $V$ restricts to a descent datum on $U$ (see the diagram (\ref{e02})).

In the general case, we can write $V$ as a finite union of open affines
$U_{i}$ such that ${}^{\sigma}U_{i}=U_{i}$ for all $\sigma\in\Gamma$. Then $V$
is the algebraic scheme over $\Omega$ obtained by patching the $U_{i}$ by
means of the maps
\begin{equation}
U_{i}\hookleftarrow U_{i}\cap U_{j}\hookrightarrow U_{j}. \label{e24}%
\end{equation}
Each intersection $U_{i}\cap U_{j}$ is again affine (\ref{dt01}), and so the
system (\ref{e24}) descends to $k$. The algebraic scheme over $k$ obtained by
patching the descended system is a model of $V$ over $k$ splitting the descent datum.
\end{proof}

\begin{corollary}
\label{dt15}Let $\Omega^{\mathrm{sep}}$ be a separable closure of $\Omega$. If
every finite set of points of $V(\Omega^{\mathrm{sep}})$ is contained in an
open affine algebraic subscheme of $V_{\Omega^{\mathrm{sep}}}$, then every
descent datum on $V$ is effective.
\end{corollary}

\begin{proof}
As we noted in the paragraph before \ref{dt10a}, an $\Omega/k$-descent datum
for $V$ extends in a natural way to an $\Omega^{\mathrm{sep}}/k$-descent datum
for $V_{\Omega^{\mathrm{sep}}}$, and if a model $(V_{0},\varphi)$ over $k$
splits the second descent datum, then it also splits the first. Thus, we may
suppose that $\Omega$ is separably closed.

Let $(\varphi_{\sigma})_{\sigma\in\Gamma}$ be a descent datum on $V$, and let
$U$ be an open subscheme of $V$. By definition, $(\varphi_{\sigma})$ is split
by a model $(V_{1},\varphi)$ of $V$ over some finite extension $k_{1}$ of $k$.
After possibly replacing $k_{1}$ with a larger finite extension, which we may
suppose to be Galois over $k$, we have that there exists an open subscheme
$U_{1}$ of $V_{1}$ such that $\varphi(U)=U_{1\Omega}$. Now \ref{dt11}b shows
that $^{\sigma}U$ depends only on the coset $\sigma\Delta$, where
$\Delta=\Gal(\Omega/k_{1})$. In particular, $\{^{\sigma}U\mid\sigma\in
\Gamma\}$ is finite, and so the scheme $U^{\prime}\overset{\df}{=}%
\bigcap_{\sigma\in\Gamma}{}^{\sigma}U$ is open in $V$. Note that (see
\ref{dt11}a)
\[
^{\tau}U^{\prime}={}^{\tau}(\bigcap_{\sigma\in\Gamma}{}^{\sigma}%
U)=(\bigcap_{\sigma\in\Gamma}{}^{\tau\sigma}U)=U^{\prime}%
\]
for all $\tau\in\Gamma$.

Let $P$ be a closed point of $V$. Because $\{{}^{\sigma}P\mid\sigma\in
\Gamma\}$ is finite, it is contained in an open affine $U$ of $V$. Now
$U^{\prime}\overset{\df}{=}\bigcap_{\sigma\in\Gamma}{}^{\sigma}U$ is an open
affine in $V$ containing $P$ and such that ${}^{\sigma}U^{\prime}=U^{\prime}$
for all $\sigma\in\Gamma$. It follows that the scheme $V$ satisfies the
hypothesis of Theorem \ref{dt14}.
\end{proof}

\begin{corollary}
\label{dt16}Descent is effective in each of the following two cases:

\begin{enumerate}
\item $V$ is quasi-projective, or

\item an affine algebraic group $G$ acts transitively on $V$.
\end{enumerate}
\end{corollary}

\begin{proof}
(a) Apply \ref{dt02}.

(b) We may assume $\Omega$ to be separably closed. Let $S$ be a finite set of
points of $V(\Omega)$, and let $U$ be an open affine in $V$. For each $P\in
S$, there is a nonempty open algebraic subscheme $G_{P}$ of $G$ such that
$G_{P}\cdot P\subset U$. Because $\Omega$ is separably closed, there exists a
$g\in(\textstyle\bigcap\nolimits_{P\in S}G_{P}\cdot P)(\Omega)$ (by
\ref{dt06}; the separable points are dense in an algebraic scheme). Now
$g^{-1}U$ is an open affine containing $S$.
\end{proof}

\section{Application: Weil restriction}

Let $K/k$ be a finite extension of fields, and let $V$ be an algebraic scheme
over $K$. A pair $(V_{\ast},\varphi)$ consisting of an algebraic scheme
$V_{\ast}$ over $k$ and a regular map $\varphi\colon V_{\ast K}\rightarrow V$
is called the $K/k$-\emph{Weil restriction }of $V$ if it has the following
universal property: for any algebraic scheme $T$ over $k$ and regular map
$\varphi^{\prime}\colon T_{K}\rightarrow V$, there exists a unique regular map
$\psi\colon T\rightarrow V$ (of $k$-scheme) such that $\varphi\circ\psi
_{K}=\varphi^{\prime}$, i.e.,
\[
\text{given} \begin{tikzcd}
T_K\arrow{rd}{\varphi'}\\
V_{\ast K}\arrow{r}{\varphi}&V
\end{tikzcd}
\text{there exists a unique} \begin{tikzcd}
T\arrow{d}{\psi} \\
V_{\ast}
\end{tikzcd}
\text{such that} \begin{tikzcd}
T_K\arrow{d}{\psi_K}\arrow{rd}{\varphi'} \\
V_{\ast K}\arrow{r}{\varphi}&V
\end{tikzcd}
\text{commutes.}
\]

\noindent In other words, $(V_{\ast},\varphi)$ is the $K/k$-Weil restriction
of $V$ if $\varphi$ defines an isomorphism%

\[
\psi\mapsto\varphi\circ\psi_{K}\colon\Mor_{k}(T,V_{\ast})\rightarrow
\Mor_{K}(T_{K},V)
\]
(natural in the $k$-algebraic scheme $T$); in particular,
\[
V_{\ast}(A)\simeq V(K\otimes_{k}A)
\]
(natural in the affine $k$-algebra $A$). If it exists, the $K/k$-Weil
restriction of $V$ is uniquely determined by its universal property (up to a
unique isomorphism).

When $(V_{\ast},\varphi)$ is the $K/k$-Weil restriction of $V$, the algebraic
scheme $V_{\ast}$ is said to have been\emph{ }obtained from $V$\emph{ by
(Weil) restriction of scalars }or \emph{by restriction of the base field}$.$

\begin{proposition}
\label{dt16a}If $V$ is a quasi-projective variety and $K/k$ is separable, then
the $K/k$-Weil restriction of $V$ exists.
\end{proposition}

\begin{proof}
Let $\Omega$ be a Galois extension of $k$ large enough to contain all
conjugates of $K$, i.e., such that $\Omega\otimes_{k}K\simeq\prod_{\tau\colon
K\rightarrow\Omega}\tau K$. Let $V^{\prime}=\prod\tau V$ --- this is an
algebraic scheme over $\Omega$. For $\sigma\in\Gal(\Omega/k)$, define
$\varphi_{\sigma}\colon\sigma V^{\prime}\rightarrow V^{\prime}$ to be the
regular map that acts on the factor $\sigma(\tau V)$ as the canonical
isomorphism $\sigma(\tau V)\simeq(\sigma\tau)V$. Then $(\varphi_{\sigma
})_{\sigma\in\Gal(\Omega/k)}$ is a descent datum, and so defines a model
$(V_{\ast},\varphi_{\ast})$ of $V^{\prime}$ over $k$.

Choose a $\tau_{0}\colon K\rightarrow\Omega$. The projection map $V^{\prime
}\rightarrow\tau_{0}V$ is invariant under the action of $\Gal(\Omega/\tau
_{0}K)$, and so defines a regular map $(V_{\ast})_{\tau_{0}K}\rightarrow
\tau_{0}V$ (\ref{dt5}), and hence a regular map $\varphi\colon V_{\ast
K}\rightarrow V$. It is easy to check that this has the correct universal property.
\end{proof}

\section{Specialization}

Let $U$ be an integral algebraic scheme over $k$, and let $\varphi\colon
V\rightarrow U$ be a dominant map. The%
\index{fibre!generic}
\emph{generic fibre }of $\varphi$ is a regular map $\varphi_{K}\colon
V_{K}\rightarrow\Spec K$, where $K=k(U)$. For example, if $V$ and $U$ are
affine, then $\varphi$ is $\Spec$ of an injective homomorphism of rings
$f\colon A\rightarrow B$, and $\varphi_{K}$ is $\Spec$ of $K\simeq
A\otimes_{k}K\rightarrow B\otimes_{k}K$, where $K$ is the field of fractions
of the integral domain $A$.

Let $K$ be a field finitely generated over $k$, and let $V_{K}$ be an
algebraic scheme over $K$. For any integral algebraic scheme $U$ over $k$ with
$k(U)=K$, there exists a dominant map $\varphi\colon V\rightarrow U$ with
generic fibre $V_{K}\rightarrow\Spec K$. For example, if $U=\Spec(A)$, where
$A$ is a finitely generated $k$-subalgebra of $K$, we only have to invert the
coefficients of some set of polynomials defining $V_{K}$. Let $P$ be a closed
point in the image of $\varphi$. Then the fibre of $V$ over $P$ is an
algebraic scheme $V(P)$ over $k(P)$, called the%
\index{specialization}
\emph{specialization }of $V$ at $P$. If $k$ is algebraically closed, then
$k(P)=k$.

\section{Rigid descent}

\begin{proposition}
\label{dt18}Let $V$ and $W$ be algebraic schemes over an algebraically closed
field $k$. If $V$ and $W$ become isomorphic over some field containing $k$,
then they are already isomorphic over $k$.
\end{proposition}

\begin{proof}
The hypothesis implies that, for some field $K$ finitely generated over $k$,
there exists an isomorphism $\varphi\colon V_{K}\rightarrow W_{K}$. Let $U$ be
an integral algebraic scheme over $k$ such that $k(U)=K$. After possibly
replacing $U$ with an open subscheme, we may extend $\varphi$ to an
isomorphism $\varphi_{U}\colon U\times V\rightarrow U\times W$. The fibre of
$\varphi_{U}$ at any closed point of $U$ is an isomorphism $V\rightarrow W$.
\end{proof}

\begin{example}
\label{dt16e}Let $\Omega\supset k$ be algebraically closed fields, and let $E$
be an elliptic curve over $\Omega$. There exists a model of $E$ over a
subfield $K$ of $\Omega$ if and only if $j(E)\in K$. Therefore, if there exist
models of $E$ over subfields $K_{1},K_{2}$ of $\Omega$ such that $K_{1}\cap
K_{2}=k$, then $E$ has a model over $k$. We now prove a similar statement for
an arbitrary algebraic scheme over $\Omega$.
\end{example}

Let $\Omega\supset k$ be fields. Subfields $K_{1}$ and $K_{2}$ of $\Omega$
containing $k$ are said to be \emph{linearly disjoint} over $k$ if the
homomorphism%
\[
\tstyle\sum a_{i}\otimes b_{i}\mapsto\sum a_{i}b_{i}\colon K_{1}\otimes
_{k}K_{2}\rightarrow K_{1}\cdot K_{2}\subset\Omega
\]
is injective.

\begin{proposition}
\label{dt16n}Let $\Omega\supset k$ be algebraically closed fields, and let $V$
be an algebraic scheme over $\Omega$. If there exist models of $V$ over
subfields $K_{1},K_{2}$ of $\Omega$ finitely generated over $k$ and linearly
disjoint over $k$, then there exists a model of $V$ over $k$.
\end{proposition}

\begin{proof}
The model of $V$ over $K_{1}$ extends to a model over an integral affine
algebraic scheme $U_{1}$ with $k(U_{1})=K_{1}$, i.e., there exists a
surjective map $V_{1}\rightarrow U_{1}$ of $k$-schemes whose generic fibre is
the model of $V$ over $K_{1}$. A similar statement applies to the model over
$K_{2}$. Because $K_{1}$ and $K_{2}$ are linearly disjoint, $K_{1}\otimes
_{k}K_{2}$ is an integral domain with field of fractions $k(U_{1}\times
U_{2})$. From the map $V_{1}\rightarrow U_{1}$, we get a map $V_{1}\times
U_{2}\rightarrow U_{1}\times U_{2}$, and similarly for $V_{2}$.

Assume initially that $V_{1}\times U_{2}$ and $U_{1}\times V_{2}$ are
isomorphic over $U_{1}\times U_{2}$, so that we have a commutative diagram,
\[
\begin{tikzcd}
V_1\arrow{d}
&V_1\times U_2\arrow{l}[swap]{\pr_1}\arrow{rd}\arrow{rr}{\approx}
&&U_1\times V_2\arrow{r}{\pr_2}\arrow{ld}
&V_2\arrow{d}\\
U_1&&U_1\times U_2\arrow{ll}\arrow{rr}&&U_2.
\end{tikzcd}
\]
Let $P$ be a closed point of $U_{1}$. When we pull back the central triangle
to the algebraic subscheme $P\times U_{2}$ of $U_{1}\times U_{2}$, we get the
diagram at left below. Note that $P\simeq\Spec k$ (because $k$ is
algebraically closed) and $P\times U_{2}\simeq U_{2}$.
\[
\begin{tikzcd}[column sep=small]
V_1(P)\times U_2\arrow{rd}\arrow{rr}{\approx}
&&P\times V_2\arrow{ld}
&V_1(P)_{K_2}\arrow{rr}{\approx}\arrow{rd}
&&V_{2K_2}\arrow{ld}\\
&P\times U_2&&&\Spec(K_2).
\end{tikzcd}
\]
The generic fibre of this diagram is the diagram at right. Here $V_{1}%
(P)_{K_{2}}$ is the algebraic scheme over $K_{2}$ obtained from $V_{1}(P)$ by
extension of scalars $k\rightarrow K_{2}$. As $V_{2K_{2}}$ is a model $V$ over
$K_{2}$, it follows that $V_{1}(P)$ is a model of $V$ over $k$.

We now prove the general case. The schemes $(V_{1}\times U_{2})_{k(U_{1}\times
U_{2})}$ and $(U_{1}\times V_{2})_{k(U_{1}\times U_{2})}$ become isomorphic
over some finite field extension $L$ of $k(U_{1}\times U_{2})$. Let $\bar{U}$
be the normalization\footnote{If $U_{1}\times U_{2}=\Spec C$, then $\bar
{U}=\Spec\bar{C}$, where $\bar{C}$ is the integral closure of $C$ in $L$.} of
$U_{1}\times U_{2}$ in $L$, and let $U$ be a dense open subset of $\bar{U}$
such that some isomorphism of $(V_{1}\times U_{2})_{L}$ with $(U_{1}\times
V_{2})_{L}$ extends to an isomorphism over $U$. Then \ref{dt03} shows that
$\bar{U}\rightarrow U_{1}\times U_{2}$ is surjective, and so the image
$U^{\prime}$ of $U$ in $U_{1}\times U_{2}$ contains a nonempty (hence dense)
open subset of $U_{1}\times U_{2}$ (see \ref{dt04}). In particular,
$U^{\prime}$ contains a subset $P\times U_{2}^{\prime}$ with $U_{2}^{\prime}$
a nonempty open subset of $U_{2}$. Now the previous argument gives us schemes
$V_{1}(P)_{K_{2}}$ and $V_{2K_{2}}$ over $K_{2}$ that become isomorphic over
$k(U^{\prime\prime})$, where $U^{\prime\prime}$ is the inverse image of
$P\times U_{2}^{\prime}$ in $\bar{U}$. As $k(U^{\prime\prime})$ is a finite
extension of $K_{2}$, this again shows that $V_{1}(P)$ is a model of $V$ over
$k$.
\end{proof}

\begin{proposition}
\label{dt17}Let $\Omega$ be algebraically closed of infinite transcendence
degree over $k$, and assume that $k$ is algebraically closed in $\Omega$. For
any $K\subset\Omega$ finitely generated over $k$, there exists a $\sigma
\in\Aut(\Omega/k)$ such that $K$ and $\sigma K$ are linearly disjoint over
$k.$
\end{proposition}

\begin{proof}
Let $a_{1},\ldots,a_{n}$ be a transcendence basis for $K/k$, and extend it to
a transcendence basis $a_{1},\ldots,a_{n},b_{1},\ldots,b_{n},\ldots$ of
$\Omega/k$. Let $\sigma$ be any permutation of the transcendence basis such
that $\sigma(a_{i})=b_{i}$ for all $i$. Then $\sigma$ defines a $k$%
-automorphism of $k(a_{1},\ldots a_{n},b_{1},\ldots,b_{n},\ldots)$, which we
extend to an automorphism of $\Omega$.

Let $K_{1}=k(a_{1},\ldots,a_{n})$. Then $\sigma K_{1}=k(b_{1},\ldots,b_{n})$,
and certainly $K_{1}$ and $\sigma K_{1}$ are linearly disjoint. Note that
$K_{1}\otimes_{k}\sigma K_{1}\subset$ $K\otimes_{k}\sigma K$ are integral
domains (by \ref{dt05}) and that $K\otimes_{k}\sigma K$ is integral over
$K_{1}\otimes_{k}\sigma K_{1}$. The kernel of $K\otimes_{k}\sigma K\rightarrow
K\cdot\sigma K$ is a prime ideal $\mathfrak{q}{}$ such that
\[
\mathfrak{q}{}\cap\left(  K_{1}\otimes_{k}\sigma K_{1}\right)  =0=\{0\}\cap
\left(  K_{1}\otimes_{k}\sigma K_{1}\right)  ,
\]
and so $\mathfrak{q}{}=0$ (by \ref{dt03}).
\end{proof}

\begin{proposition}
\label{dt20}Let $\Omega\supset k$ be algebraically closed fields, and let $V$
be an algebraic scheme over $\Omega$. If $V$ is isomorphic to $\sigma V$ for
every $\sigma\in\Aut(\Omega/k)$, then $V$ has a model over $k$.
\end{proposition}

\begin{proof}
After replacing $\Omega$ with a larger algebraically closed field, we may
suppose that it has infinite transcendence degree over $k$. There exists a
model $(V_{0},\varphi)$ of $V$ over a subfield $K$ of $\Omega$ finitely
generated over $k$. According to Proposition \ref{dt17}, there exists a
$\sigma\in\Aut(\Omega/k)$ such that $K$ and $\sigma K$ are linearly disjoint.
Now
\[
(\sigma V_{0},(\sigma V_{0})_{\Omega}=\sigma(V_{0\Omega})\overset{\sigma
\varphi}{\longrightarrow}\sigma V\approx V)
\]
is a model of $V$ over $\sigma K$, and so we can apply Proposition \ref{dt16n}.
\end{proof}

In the next two theorems, $\Omega\supset k$ is an algebraically closed field
containing a perfect field (so $k=\Omega^{\Gamma}$, $\Gamma=\Aut(\Omega/k)$).

\begin{theorem}
\label{dt21}Let $V$ be a quasi-projective scheme over $\Omega$, and let
$(\varphi_{\sigma})_{\sigma\in\Gamma}$ be an $\Omega/k$-descent system for
$V$. If the only automorphism of $V$ is the identity map, then $V$ has a model
over $k$ splitting $(\varphi_{\sigma})_{\sigma}$.
\end{theorem}

\begin{proof}
According to Proposition \ref{dt20}, $V$ has a model $(V_{0},\varphi)$ over
the algebraic closure $k^{\mathrm{al}}$ of $k$ in $\Omega$, which (see
\ref{dt13m}) splits $(\varphi_{\sigma})_{\sigma\in\Aut(\Omega/k^{\mathrm{al}%
})}$.

Now $\varphi_{\sigma}^{\prime}\overset{\df}{=}\varphi^{-1}\circ\varphi
_{\sigma}\circ\sigma\varphi$ is stable under $\Aut(\Omega/k^{\mathrm{al}})$,
and hence is defined over $k^{\mathrm{al}}$ (\ref{dt5}). Moreover,
$\varphi_{\sigma}^{\prime}$ depends only on the restriction of $\sigma$ to
$k^{\mathrm{al}}$, and $(\varphi_{\sigma}^{\prime})_{\sigma\in
\Gal(k^{\mathrm{al}}/k)}$ is a descent system for $V_{0}$. It is continuous by
Proposition \ref{dt13}, and so $V_{0}$ has a model $(V_{00},\varphi^{\prime})$
over $k$ splitting $(\varphi_{\sigma}^{\prime})_{\sigma\in\Gal(k^{\mathrm{al}%
}/k)}$. Now $(V_{00},\varphi\circ\varphi_{\Omega}^{\prime})$ splits
$(\varphi_{\sigma})_{\sigma\in\Aut(\Omega/k)}$.
\end{proof}

We now consider pairs $(V,S)$, where $V$ is an algebraic scheme over $\Omega$
and $S=(P_{i})_{1\leq i\leq n}$ is a family of closed points on $V$. A
morphism $(V,(P_{i})_{1\leq i\leq n})\rightarrow(W,(Q_{i})_{1\leq i\leq n})$
is a regular map $\varphi\colon V\rightarrow W$ such that $\varphi
(P_{i})=Q_{i}$ for all $i$.

\begin{theorem}
\label{dt21m}Let $V$ be a quasi-projective scheme over $\Omega$, and let
$(\varphi_{\sigma})_{\sigma\in\Aut(\Omega/k)}$ be a descent system for $V$.
Let $S=(P_{i})_{1\leq i\leq n}$ be a finite set of points of $V$ such that

\begin{enumerate}
\item the only automorphism of $V$ fixing each $P_{i}$ is the identity map, and

\item there exists a subfield $K$ of $\Omega$ finitely generated over $k$ such
that ${}^{\sigma}P=P$ for all $\sigma\in\Gamma$ fixing $K$.
\end{enumerate}

\noindent Then $V$ has a model over $k$ splitting $(\varphi_{\sigma})$.
\end{theorem}

\begin{proof}
The preceding propositions hold with $V$ replaced by $(V,S)$ (with the same
proofs), and so the proof of Theorem \ref{dt21} applies.
\end{proof}

\begin{example}
\label{dt21n}Theorem \ref{dt21m} sometimes allows us to construct objects over
subfields of $\mathbb{C}{}$ by working entirely over $\mathbb{C}{}$. We
illustrate this with the Jacobian variety of a complete smooth curve. For such
a curve $C$ over $\mathbb{C}{}$, the complex torus%
\[
J(C)(\mathbb{C}{})=\Gamma(C,\Omega^{1})^{\vee}/H_{1}(C,\mathbb{Z}{})\text{.}%
\]
has a unique structure of a projective algebraic variety (hence of an abelian
variety). Let $P\in C(\mathbb{C}{})$. The Abel--Jacobi map $Q\mapsto\lbrack
Q-P]\colon C(\mathbb{C}{})\rightarrow J(C)(\mathbb{C}{})$ arises from a
(unique) regular map $f^{P}\colon C\rightarrow J(C)$. This has the following
universal property:\footnote{It suffices to check this in the complex-analytic
category.} for any regular map $f\colon C\rightarrow A$ from $C$ to an abelian
variety sending $P$ to $0$, there is a unique homomorphism $\phi\colon
J\rightarrow A$ such that $\phi\circ f^{P}=f$.

Now let $C$ be a complete smooth curve over a subfield $k$ of $\mathbb{C}{}$.
From $C$, we get a curve $\bar{C}$ over $\mathbb{C}{}$ and a descent datum
$(\varphi_{\sigma})_{\sigma\in\Aut(\mathbb{C}{}/k)}$. Let $J(\bar{C})$ denote
the Jacobian variety of $\bar{C}$, and let $f\colon\bar{C}\rightarrow
J(\bar{C})$ be the Abel--Jacobi map defined by some point in $C(\mathbb{C}{}%
)$. For each $\sigma\in\Aut(\mathbb{C}/k)$, there is a unique isomorphism
$\phi_{\sigma}\colon\sigma J\rightarrow J$ such that
\[
\begin{tikzcd}
\sigma\bar{C}\arrow{r}{\varphi_\sigma}\arrow{d}{\sigma f}&\bar{C}\arrow{d}{f}\\
\sigma J\arrow{r}{\phi_\sigma}&J
\end{tikzcd}
\]
commutes up to a translation (apply the universality of $(J,f)$ to get
$\phi_{\sigma}^{-1}$ and $\sigma\phi_{\sigma})$. The family $(\phi_{\sigma
})_{\sigma}$ is a descent system for $J$, and if we take $S$ to be the set of
points of order $3$ on $J(C)$, then the conditions of the theorem are
satisfied. For (a), this is proved, for example, in \cite{milneAV}, 17.5. For
(b), we may suppose (after possibly extending $k$) that $C(k)$ is nonempty,
say, $P\in C(k)$. When we set $f=f^{P}$, the above diagram commutes exactly,
and $f({}^{\sigma}Q)={}^{\sigma}f(Q)$. According to the Jacobi inversion
theorem, the map
\[
\sum m_{i}Q_{i}\mapsto\sum m_{i}f(Q_{i})\colon\Div^{0}(\bar{C})\rightarrow
J(\bar{C})(\mathbb{C})
\]
is surjective. Now $K$ can be taken to be any finitely generated field such
that the subgroup of $J(\bar{C})(\mathbb{C})$ generated $\{$ $f(Q)\mid Q\in
C(K)\}$ contains all elements of order $3$.

We can therefore define the Jacobian of $C$ over $k$ to be the model of
$J(\bar{C})$ over $k$ splitting $(\phi_{\sigma})_{\sigma}$.
\end{example}

\begin{aside}
The Theorem \ref{dt21m} is Corollary 1.2 of \cite{milne1999}, where it was
used to show that the conjecture of Langlands on the conjugation of Shimura
varieties (a statement about Shimura varieties over $\mathbb{C}$) implies the
existence of canonical models (Shimura's conjecture). There it was deduced
from Weil's theorems (see below). The present more elementary proof was
suggested by Wolfart's elementary proof of the `obvious' part of Belyi's
theorem (\cite{wolfart1997}; see also \cite{derome2003}).
\end{aside}

\section{Restatement in terms of group actions}

In this subsection, $\Omega\supset k$ are fields with $k$ perfect and $\Omega$
algebraically closed (so $k=\Omega^{\Gamma}$, $\Gamma=\Gal(\Omega/k)$). Recall
that for any algebraic variety $V$ over $k$, there is a natural action of
$\Gamma$ on $V(\Omega)$. In this subsection, we describe the essential image
of the functor%
\[
\{\text{quasi-projective varieties over }k\text{\}}\rightarrow
\{\text{quasi-projective varieties over }\Omega+\text{action of }\Gamma\}.
\]
In other words, we determine which pairs $(V,\ast),$ with $V$ a
quasi-projective variety over $\Omega$ and $\ast$ an action of $\Gamma$ on
$V(\Omega{})$,%
\[
(\sigma,P)\mapsto\sigma\ast P\colon\Gamma\times V(\Omega)\rightarrow
V(\Omega),
\]
arise from an algebraic variety over $k$. There are two obvious necessary
conditions for this.

\subsubsection{Regularity condition}

Obviously, the action should recognize that $V(\Omega)$ is not just a set, but
rather the set of points of an algebraic variety. For $\sigma\in\Gamma$, let
$\sigma V$ be the algebraic variety obtained by applying $\sigma$ to the
coefficients of the equations defining $V$, and for $P\in V(\Omega)$ let
$\sigma P$ be the point on $\sigma V$ obtained by applying $\sigma$ to the
coordinates of $P$.

\begin{definition}
\label{dt26}We say that the action $\ast$ is%
\index{action!regular}
\emph{regular} if the map
\[
\sigma P\mapsto\sigma\ast P\colon(\sigma V)(\Omega)\rightarrow V(\Omega)
\]
is a regular isomorphism for all $\sigma$.
\end{definition}

A priori, this is only a map of sets. The condition requires that it be
induced by a regular map $\varphi_{\sigma}\colon\sigma V\rightarrow V$. If
$V=V_{0\Omega}$ for some algebraic variety $V_{0}$ defined over $k$, then
$\sigma V=V$, and $\varphi_{\sigma}$ is the identity map, and so the condition
is clearly necessary.

When $V$ is affine, $V=\Spec A$, then $\ast$ is regular if and only if it
induces an action%
\[
(\sigma\ast f)(\sigma\ast P)=\sigma(f(P))
\]
of $\Gamma$ on $A$ by semilinear automorphisms.

\begin{remark}
\label{dt27}The maps $\varphi_{\sigma}$ satisfy the cocycle condition
$\varphi_{\sigma}\circ\sigma\varphi_{\tau}=\varphi_{\sigma\tau}$. In
particular, $\varphi_{\sigma}\circ\sigma\varphi_{\sigma^{-1}}=\id$, and so if
$\ast$ is regular, then each $\varphi_{\sigma}$ is an isomorphism, and the
family $(\varphi_{\sigma})_{\sigma\in\Gamma}$ is a descent system. Conversely,
if $(\varphi_{\sigma})_{\sigma\in\Gamma}$ is a descent system, then
\[
\sigma\ast P=\varphi_{\sigma}(\sigma P)
\]
defines a regular action of $\Gamma$ on $V(\Omega)$. Note that if
$\ast\leftrightarrow(\varphi_{\sigma})$, then $\sigma\ast P={}^{\sigma}P$.
\end{remark}

\subsubsection{Continuity condition}

\begin{definition}
\label{dt28}We say that the action $\ast$ is%
\index{action!continuous}
\emph{continuous} if there exists a subfield $L$ of $\Omega$ finitely
generated over $k$ and a model $V_{0}$ of $V$ over $L$ such that the action of
$\Gamma(\Omega/L)$ is that defined by $V_{0}$.
\end{definition}

For an affine algebraic variety $V$, an action of $\Gamma$ on $V$ gives an
action of $\Gamma$ on $\Omega\lbrack V]$, and one action is continuous if and
only if the other is.

Continuity is obviously necessary. It is easy to write down regular actions
that fail it, and hence do not arise from varieties over $k$.

\begin{example}
\label{dt29}The following are examples of actions that fail the continuity
condition (the second two are regular).

\begin{enumerate}
\item Let $V=\mathbb{A}{}^{1}$ and let $\ast$ be the trivial action.

\item Let $\Omega/k=\mathbb{Q}{}^{\text{al}}/\mathbb{Q}{}$, and let $N$ be a
normal subgroup of finite index in $\Gal(\mathbb{Q}{}^{\text{al}}/\mathbb{Q})$
that is not open,\footnote{For a proof that such subgroups exist, see, for
example, \cite{milneFT}, 7.29.} i.e., that fixes no extension of $\mathbb{Q}%
{}$ of finite degree. Let $V$ be the zero-dimensional algebraic variety over
$\mathbb{Q}{}^{\text{al}}$ with $V(\mathbb{Q}{}^{\text{al}})=\Gal(\mathbb{Q}%
{}^{\text{al}}/\mathbb{Q}{})/N$ equipped with its natural action.

\item Let $k$ be a finite extension of $\mathbb{Q}{}_{p}$, and let
$V=\mathbb{A}{}^{1}$. The homomorphism $k^{\times}\rightarrow\Gal(k^{\text{ab}%
}/k)$ can be used to twist the natural action of $\Gamma$ on $V(\Omega)$.
\end{enumerate}
\end{example}

\subsubsection{Restatement of the main theorems}

Recall that $\Omega\supset k$ are fields with $k$ perfect and $\Omega$
algebraically closed (so $k=\Omega^{\Gamma}$, $\Gamma=\Gal(\Omega/k)$).

\begin{theorem}
\label{dt31}Let $V$ be a quasi-projective algebraic variety over $\Omega$, and
let $\ast$ be a regular action of $\Gamma$ on $V(\Omega)$. Let $S=(P_{i}%
)_{1\leq i\leq n}$ be a finite set of points of $V$ such that

\begin{enumerate}
\item the only automorphism of $V$ fixing each $P_{i}$ is the identity map, and

\item there exists a subfield $K$ of $\Omega$ finitely generated over $k$ such
that ${}\sigma\ast P=P$ for all $\sigma\in\Gamma$ fixing $K$.
\end{enumerate}

\noindent Then $\ast$ arises from a model of $V$ over $k$.
\end{theorem}

\begin{proof}
This a restatement of Theorem \ref{dt21m}.
\end{proof}

\begin{theorem}
\label{dt30}Let $V$ be a quasi-projective algebraic variety over $\Omega$ with
an action $\ast$ of $\Gamma$. If $\ast$ is regular and continuous, then $\ast$
arises from a model of $V$ over $k$ in each of the following cases:

\begin{enumerate}
\item $\Omega$ is algebraic over $k$, or

\item $\Omega$ is has infinite transcendence degree over $k$.
\end{enumerate}
\end{theorem}

\begin{proof}
(a) Restatement of \ref{dt14}, \ref{dt16}

(b) Restatement of \ref{dt25} below (which depends on Weil's theorem
\ref{dt25b}).
\end{proof}

The condition \textquotedblleft quasi-projective\textquotedblright\ is
necessary, because otherwise the action may not stabilize enough open affine
subsets to cover $V$. In fact, an example shows that if $V$ is not
quasi-projective, then $V_{0}$ need not exist, unless it is allowed to be an
algebraic space in the sense of~Artin (see, for example, p.\thinspace131 of
Dieudonn\'{e}, J., Fondements de la G\'{e}om\'{e}trie Alg\'{e}brique Moderne,
Presse de l'Universit\'{e} de Montr\'{e}al, 1964).

\section{Faithfully flat descent}

Recall that a homomorphism $f\colon A\rightarrow B$ of rings is flat if the
functor \textquotedblleft extension of scalars\textquotedblright%
\ $M\rightsquigarrow B\otimes_{A}M$ is exact. It is \emph{faithfully flat }if
a sequence%
\[
0\rightarrow M^{\prime}\rightarrow M\rightarrow M^{\prime\prime}\rightarrow0
\]
of $A$-modules is exact if and only if
\[
0\rightarrow B\otimes_{A}M^{\prime}\rightarrow B\otimes_{A}M\rightarrow
B\otimes_{A}M^{\prime\prime}\rightarrow0
\]
is exact. For a field $k$, a homomorphism $k\rightarrow A$ is always flat
(because exact sequences of $k$-vector spaces are split-exact), and it is
faithfully flat if $A\neq0$.

The next theorem and its proof are quintessential Grothendieck.

\begin{theorem}
\label{dt22}If $f\colon A\rightarrow B$ is faithfully flat, then the sequence%
\[
0\rightarrow A\overset{f}{\longrightarrow}B\overset{d^{0}}{\longrightarrow
}B^{\otimes2}\rightarrow\cdots\rightarrow B^{\otimes r}\overset{d^{r-1}%
}{\longrightarrow}B^{\otimes r+1}\rightarrow\cdots
\]
is exact, where%
\begin{align*}
B^{\otimes r}  &  =B\otimes_{A}B\otimes_{A}\cdots\otimes_{A}B\qquad
\text{(}r\text{ times)}\\
d^{r-1}  &  =\tstyle\sum(-1)^{i}e_{i}\\
e_{i}(b_{0}\otimes\cdots\otimes b_{r-1})  &  =b_{0}\otimes\cdots\otimes
b_{i-1}\otimes1\otimes b_{i}\otimes\cdots\otimes b_{r-1}.
\end{align*}

\end{theorem}

\begin{proof}
It is easily checked that $d^{r}\circ d^{r-1}=0$. We assume first that $f$
admits a section, i.e., that there is a homomorphism $g\colon B\rightarrow A$
such that $g\circ f=1$, and we construct a contracting homotopy $k_{r}\colon
B^{\otimes r+2}\rightarrow B^{\otimes r+1}$. Define%
\[
k_{r}(b_{0}\otimes\cdots\otimes b_{r+1})=g(b_{0})b_{1}\otimes\cdots\otimes
b_{r+1},\qquad r\geq-1.
\]
It is easily checked that
\[
k_{r+1}\circ d^{r+1}+d^{r}\circ k_{r}=1,\quad r\geq-1\text{,}%
\]
and this shows that the sequence is exact.

Now let $A^{\prime}$ be an $A$-algebra. Let $B^{\prime}=A^{\prime}\otimes
_{A}B$ and let $f^{\prime}=1\otimes f\colon A^{\prime}\rightarrow B^{\prime}$.
The sequence corresponding to $f^{\prime}$ is obtained from the sequence for
$f$ by tensoring with $A^{\prime}$ (because $B^{\otimes r}\otimes A^{\prime
}\cong B^{\prime\otimes r}$ etc.). Thus, if $A^{\prime}$ is a faithfully flat
$A$-algebra, it suffices to prove the theorem for $f^{\prime}$. Take
$A^{\prime}=B$, and then $b\overset{f}{\mapsto}b\otimes1\colon B\rightarrow
B\otimes_{A}B$ has a section, namely, $g(b\otimes b^{\prime})=bb^{\prime}$,
and so the sequence is exact.
\end{proof}

\begin{theorem}
\label{dt23}If $f\colon A\rightarrow B$ is faithfully flat and $M$ is an
$A$-module, then the sequence%
\[
0\rightarrow M\overset{1\otimes f}{\longrightarrow}M\otimes_{A}%
B\overset{1\otimes d^{0}}{\longrightarrow}M\otimes_{A}B^{\otimes2}%
\rightarrow\cdots\rightarrow M\otimes_{B}B^{\otimes r}\overset{1\otimes
d^{r-1}}{\longrightarrow}B^{\otimes r+1}\rightarrow\cdots
\]
is exact.
\end{theorem}

\begin{proof}
As in the above proof, one may assume that $f$ has a section, and use it to
construct a contracting homotopy.
\end{proof}

\begin{remark}
\label{dt23a}Let $f\colon A\rightarrow B$ be a faithfully flat homomorphism,
and let $M$ be an $A$-module. Write $M^{\prime}$ for the $B$-module $f_{\ast
}M=B\otimes_{A}M$. The module $e_{0\ast}M^{\prime}=(B\otimes_{A}B)\otimes
_{B}M^{\prime}$ may be identified with $B\otimes_{A}M^{\prime}$, where
$B\otimes_{A}B$ acts by $(b_{1}\otimes b_{2})(b\otimes m)=b_{1}b\otimes
b_{2}m$, and $e_{1\ast}M^{\prime}$ may be identified with $M^{\prime}%
\otimes_{A}B$, where $B\otimes_{A}B$ acts by $(b_{1}\otimes b_{2})(m\otimes
b)=b_{1}m\otimes b_{2}b$. There is a canonical isomorphism $\phi\colon
e_{1\ast}M^{\prime}\rightarrow e_{0\ast}M^{\prime}$ arising from%
\[
e_{1\ast}M^{\prime}=(e_{1}f)_{\ast}M=(e_{0}f)_{\ast}M=e_{0\ast}M^{\prime};
\]
explicitly, it is the map%
\[
(b\otimes m)\otimes b^{\prime}\mapsto b\otimes(b^{\prime}\otimes m)\colon
M^{\prime}\otimes_{A}B\rightarrow B\otimes_{A}M^{\prime}.
\]
Moreover, $M$ can be recovered from the pair $(M^{\prime},\phi)$ because%
\[
M=\{m\in M^{\prime}\mid1\otimes m=\phi(m\otimes1)\}.
\]
Conversely, every pair $(M^{\prime},\phi)$ satisfying certain obvious
conditions does arise in this way from an $A$-module. Given $\phi\colon
M^{\prime}\otimes_{A}B\rightarrow B\otimes_{A}M^{\prime}$, define%
\begin{align*}
\phi_{1}\colon B\otimes_{A}M^{\prime}\otimes_{A}B  &  \rightarrow B\otimes
_{A}B\otimes_{A}M^{\prime}\\
\phi_{2}\colon M^{\prime}\otimes_{A}B\otimes_{A}B  &  \rightarrow B\otimes
_{A}B\otimes_{A}M^{\prime},\\
\phi_{3}\colon M^{\prime}\otimes_{A}B\otimes_{A}B  &  \rightarrow B\otimes
_{A}M^{\prime}\otimes_{A}B
\end{align*}
by tensoring $\phi$ with $\id_{B}$ in the first, second, and third positions
respectively. Then a pair $(M^{\prime},\phi)$ arises from an $A$-module $M$ as
above if and only if $\phi_{2}=\phi_{1}\circ\phi_{3}$. The necessity is easy
to check. For the sufficiency, define%
\[
M=\{m\in M^{\prime}\mid1\otimes m=\phi(m\otimes1)\}.
\]
There is a canonical map $b\otimes m\mapsto bm\colon B\otimes_{A}M\rightarrow
M^{\prime}$, and it suffices to show that this is an isomorphism (and that the
map arising from $M$ is $\phi$). Consider the diagram
\[
\begin{tikzcd}[column sep=large]
M^{\prime}\otimes_{A}B\arrow{d}{\phi}
\arrow[r, shift left=0.6ex,"\alpha\otimes\id_B"]
\arrow[r,shift right=0.6ex,"\beta\otimes\id_B"']
& B\otimes_{A}M^{\prime}\otimes_{A}B\arrow{d}{\phi_1}\\
B\otimes_{A}M^{\prime}
\arrow[r, shift left=0.6ex,"e_0\otimes\id_{M^{\prime}}"]
\arrow[r,shift right=0.6ex,"e_1\otimes\id_{M^{\prime}}"']
& B\otimes_{A}B\otimes_{A}M^{\prime}%
\end{tikzcd}
\]
in which $\alpha(m)=1\otimes m$ and $\beta(m)=\phi(m)\otimes1$. As the diagram
commutes with either the upper of the lower horizontal maps (for the lower
maps, this uses the relation $\phi_{2}=\phi_{1}\circ\phi_{3}$), $\phi$ induces
an isomorphism on the kernels. But, by definition of $M$, the kernel of the
pair $(\alpha\otimes1,\beta\otimes1)$ is $M\otimes_{A}B$, and, according to
(\ref{dt23}), the kernel of the pair $(e_{0}\otimes1,e_{1}\otimes1)$ is
$M^{\prime}$. This completes the proof.
\end{remark}

\begin{theorem}
\label{dt23c}Let $f\colon A\rightarrow B$ be a faithfully flat homomorphism.
Let $R$ be a $B$-algebra and $\phi\colon R\otimes_{A}B\rightarrow B\otimes
_{A}R$ a homomorphism of $B$-algebras. There exists an $A$-algebra $R_{0}$ and
an isomorphism $\varphi\colon B\otimes_{A}R_{0}\rightarrow R$ such that
$\phi=(\id_{B}\otimes\varphi)\circ(\varphi\otimes\id_{B})^{-1}$ if and only if
(with the above notation)%
\[
\phi_{2}=\phi_{1}\circ\phi_{3}.
\]
Moreover, when this is so, the pair $(R,\varphi)$ is unique up to a unique
isomorphism, and $R_{0}$ is finitely generated if $R$ is finitely generated.
\end{theorem}

\begin{proof}
When $M$ is a $B$-module, we proved this in \ref{dt23a}. The same argument
applies to algebras.
\end{proof}

A morphism $p\colon W\rightarrow V$ of schemes is \emph{faithfully flat} if it
is surjective on the underlying sets and $\mathcal{O}{}_{\varphi
(P)}\rightarrow\mathcal{O}{}_{P}$ is flat for all $P\in W$.

\begin{theorem}
\label{dt24}Let $p\colon W\rightarrow V$ be a faithfully flat map of schemes.
Let $U$ be a scheme quasi-projective over $W$ and $\phi\colon\pr_{1}^{\ast
}U\rightarrow\pr_{2}^{\ast}U$ an isomorphism of $W\times_{V}W$-schemes. There
exists a scheme $U_{0}$ over $V$ and an isomorphism $\varphi_{0}\colon
p^{\ast}U_{0}\rightarrow U$ such that $\phi=\pr_{2}^{\ast}(\varphi_{0}%
)\circ\pr_{1}^{\ast}(\varphi_{0})^{-1}$ if and only if
\[
\pr_{31}^{\ast}(\phi)=\pr_{32}^{\ast}(\phi)\circ\pr_{21}^{\ast}(\phi).
\]
Moreover, when this is so, the pair $(V,\varphi_{0})$ is unique up to a unique
isomorphism, and $V_{0}$ is quasi-projective.
\end{theorem}

\noindent Here $\pr_{1}$ and $\pr_{2}$ denote the projections $W\times
_{V}W\rightarrow W$ and $\pr_{ji}$ denotes the projection $W\times_{V}%
W\times_{V}W\rightarrow W\times_{V}W$ such that $p_{ji}(w_{1},w_{2}%
,w_{3})=(w_{j},w_{i})$.

\begin{proof}
When $U$, $V$, and $W$ are affine, this becomes the statement \ref{dt23c}. We
omit the proof of the extension to the general case.
\end{proof}

\begin{example}
\label{dt32}Let $\Gamma$ be a finite group, viewed as an algebraic group over
$k$ of dimension $0$. Let $V$ be an algebraic scheme over $k$. A scheme
\emph{Galois over }$V$ \emph{with Galois group }$\Gamma$ is a finite morphism
$W\rightarrow V$ of $k$-schemes together with a morphism $W\times
\Gamma\rightarrow W$ such that

\begin{enumerate}
\item for all $k$-algebras $R$, $W(R)\times\Gamma(R)\rightarrow W(R)$ is an
action of the group $\Gamma(R)$ on the set $W(R)$ in the usual sense, and the
map $W(R)\rightarrow V(R)$ is compatible with the action of $\Gamma(R)$ on
$W(R)$ and its trivial action on $V(R)$, and

\item the morphism $(w,\sigma)\mapsto(w,w\sigma)\colon W\times\Gamma
\rightarrow W\times_{V}W$ is an isomorphism.
\end{enumerate}

Then there is a commutative diagram
\[
\begin{tikzcd}
V\arrow[equals]{d}&W\arrow{l}\arrow[equals]{d}
&W\times\Gamma\arrow{d}{\simeq}
\arrow[l, shift left=0.45ex]
\arrow[l,shift right=0.45ex]
&W\times \Gamma\times \Gamma\arrow{d}{\simeq}
\arrow[l, shift left=0.9ex]
\arrow[l]
\arrow[l,shift right=0.9ex]\\
V&W\arrow{l}&W\times_V W
\arrow[l, shift left=0.45ex]
\arrow[l,shift right=0.45ex]
&W\times_V W\times_V W
\arrow[l, shift left=0.9ex]
\arrow[l]
\arrow[l,shift right=0.9ex]
\end{tikzcd}
\]
in which the vertical isomorphisms are%
\begin{align*}
(w,\sigma)  &  \mapsto(w,w\sigma)\\
(w,\sigma_{1},\sigma_{2})  &  \mapsto(w,w\sigma_{1},w\sigma_{1}\sigma_{2}).
\end{align*}
Therefore, in this case, Theorem \ref{dt24} says that to give an algebraic
scheme affine over $V$ is the same as giving an algebraic scheme affine over
$W$ together with an action of $\Gamma$ on it compatible with that on $W$.
When we take $W$ and $V$ to be the spectra of fields, then this becomes the
affine case of Theorem \ref{dt14}.
\end{example}

\subsection{Noncommutative rings}

\begin{definition}
\label{dt41}Let $f\colon A\rightarrow B$ be a homomorphism of rings, not
necessarily commutative, such that $B$ is a faithfully flat as a left
$A$-module. A \emph{descent datum} on a right $B$-module $M$ is a homomorphism
of right $B$-modules $\rho_{M}\colon M\rightarrow M\otimes_{A}B$ such that the
two composed maps
\[
\begin{tikzcd}[column sep=large]
M\arrow{r}{\rho_M}&M\otimes_A B
\arrow[r,shift left=0.6ex,"\rho_M\otimes B"]
\arrow[r,shift right=0.6ex,"m\otimes b\mapsto m\otimes 1_B\otimes b"']
&M\otimes_A B\otimes_A B
\end{tikzcd}
\]
are equal and the map
\[
\begin{tikzcd}[column sep=large]
M\arrow{r}{\rho_Y}&M\otimes_A B\arrow{r}{m\otimes b\mapsto mb}&M
\end{tikzcd}
\]
equals the identity map.
\end{definition}

With the obvious notion of morphism, the pairs $(M,\rho)$ consisting of a
right $B$-module and a descent datum form a category $\mathsf{Desc}(B/A)$.

\begin{theorem}
[Faithfully flat descent]\label{dt42}The functor%
\[
\Phi\colon\Mod_{A}\rightarrow\mathsf{Desc}(B/A),\quad M\rightsquigarrow
(M\otimes_{A}B,\rho_{M}),\quad\rho_{M}(m\otimes b)=m\otimes1\otimes b
\]
is an equivalence of categories.
\end{theorem}

This follows from the next more precise statement.

\begin{lemma}
\label{dt43}Let $(N,\rho_{N})$ be a right $B$-module equipped with a descent
datum. Then
\[
N^{\prime}\overset{\df}{=}\{y\in N\mid\rho_{N}(n)=n\otimes1\}
\]
is an $A$-submodule of $N$ such that
\[
N^{\prime}\otimes_{A}B\simeq N.
\]

\end{lemma}

\begin{proof}
This follows from the comonadicity theorem in category theory. See, for
example, \cite{deligne1990}, Proposition 4.4.
\end{proof}

When the rings are commutative, it is possible to show that descent data in
the above sense correspond to descent data in the commutative sense. This
gives a different approach to faithfully flat descent for commutative rings,
which, however, is not simpler than the direct approach (\ref{dt23a}).

\section{Weil's descent theorems}

In this section, $\Omega$ is an algebraically closed field containing the
field $k$. We let $k^{\mathrm{sep}}$ denote the separable closure of $k$ in
$\Omega$. The next statement is essentially Theorem 3 of \cite{weil1956}.

\begin{theorem}
\label{dt25a}Let $K$ be a finite separable extension of $k$, and let $I$ be
the set of $k$-homo\-morphisms $K\rightarrow\Omega$. Let $V$ be a
quasi-projective algebraic scheme over $K$, and for each pair $(\sigma,\tau)$
of elements of $I$, let $\phi_{\tau,\sigma}$ be an isomorphism $\sigma
V\rightarrow\tau V$ (of algebraic schemes over $\Omega$). Then there exists an
algebraic scheme $V_{0}$ over $k$ and an isomorphism $\phi\colon
V_{0K}\rightarrow V$ such that $\phi_{\tau,\sigma}=\tau\phi\circ(\sigma
\phi)^{-1}$ for all $\sigma,\tau\in I$ if and only if the $\phi_{\tau,\sigma}$
are defined over $k^{\mathrm{sep}}$ and satisfy the following conditions,

\begin{enumerate}
\item $\phi_{\tau,\rho}=\phi_{\tau,\sigma}\circ\phi_{\sigma,\rho}$ for all
$\rho,\sigma,\tau\in I$;

\item $\phi_{\tau\omega,\sigma\omega}=\omega\phi_{\tau,\sigma}$ for all
$\sigma,\tau\in I$ and all $k_{0}$-automorphisms $\omega$ of $k_{0}%
^{\mathrm{al}}$ over $k_{0}$.
\end{enumerate}

\noindent Moreover, when this is so, the pair $(V_{0},\phi)$ is unique up to
isomorphism over $k_{0}$, and $V_{0}$ is quasi-projective or quasi-affine if
$V$ is.
\end{theorem}

\begin{proof}
The conditions are obviously necessary. For the converse, fix an embedding
$i\colon K\rightarrow k^{\mathrm{sep}}$. Then the isomorphisms $\phi
_{\sigma,\tau}$ define a descent datum on $iV$, and Corollary \ref{dt16}
provides us with a pair $(V_{0},\phi)$ satisfying the required conditions (and
$(V_{0},\phi)$ is unique up to a \textit{unique }isomorphism over $k_{0}$).
\end{proof}

An extension $K$ of a field $k$ is said to be \emph{regular }if it is finitely
generated, admits a separating transcendence basis, and $k$ is algebraically
closed in $K$. These are precisely the fields that arise as the field of
rational functions on a geometrically irreducible algebraic variety over $k.$

Let $k$ be a field, and let $k(t)$, $t=(t_{1},\ldots,t_{n}),$ be a regular
extension of $k$ (in Weil's terminology, $t$ is a \textit{generic point} of an
algebraic variety over $k$). By $k(t^{\prime})$ we shall mean a field
isomorphic to $k(t)$ by $t\mapsto t^{\prime}$, and we write $k(t,t^{\prime})$
for the field of fractions of $k(t)\otimes_{k}k(t^{\prime})$.\footnote{If
$k(t)$ and $k(t^{\prime})$ are linearly disjoint subfields of $\Omega$, then
$k(t,t^{\prime})$ is the subfield of $\Omega$ generated over $k$ by $t$ and
$t^{\prime}$.} When $V_{t}$ is an algebraic scheme over $k(t)$, we shall write
$V_{t^{\prime}}$ for the algebraic scheme over $k(t^{\prime})$ obtained from
$V_{t}$ by base change with respect to $t\mapsto t^{\prime}\colon
k(t)\rightarrow k(t^{\prime})$. Similarly, if $f_{t}$ denotes a regular map of
schemes over $k(t)$, then $f_{t^{\prime}}$ denotes the regular map over
$k(t^{\prime})$ obtained by base change. Similarly, $k(t^{\prime\prime})$ is a
second field isomorphic to $k(t)$ by $t\mapsto t^{\prime\prime}$ and
$k(t,t^{\prime},t^{\prime\prime})$ is the field of fractions of $k(t)\otimes
_{k}k(t^{\prime})\otimes_{k}k(t^{\prime\prime})$.

The next statement is essentially Theorem 6 and Theorem 7 of \cite{weil1956}.

\begin{theorem}
\label{dt25b}With the above notation, let $V_{t}$ be a quasi-projective
scheme over $k(t)$, and, for each pair $(t,t^{\prime})$, let
$\phi_{t^{\prime},t}$ be an isomorphism $V_{t}\rightarrow V_{t^{\prime}}$
defined over $k(t,t^{\prime})$. Then there exists an algebraic scheme $V$
defined over $k$ and an isomorphism $\phi_{t}\colon V_{k(t)}\rightarrow V_{t}$
(of schemes over $k(t)$) such that $\phi_{t^{\prime},t}=\phi_{t^{\prime}}%
\circ\phi_{t}^{-1}$ if and only if $\phi_{t^{\prime},t}$ satisfies the
following condition,%
\[
\phi_{t^{\prime\prime},t}=\phi_{t^{\prime\prime},t^{\prime}}\circ
\phi_{t^{\prime},t}\quad\text{(isomorphism of schemes over }k(t,t^{\prime
},t^{\prime\prime}).
\]
Moreover, when this is so, the pair $(V,\phi_{t})$ is unique up to an
isomorphism over $k$, and $V$ is quasi-projective or quasi-affine if $V$ is.
\end{theorem}

\begin{proof}
The condition is obviously necessary. Assume initially that $V_{t}$ is affine,
and let $R_{r}=\mathcal{O}{}(V_{t})$. From $\phi_{t,t^{\prime}}$ we get a
commutative diagram%
\[
\begin{tikzcd}
R_{t}\otimes_{k(t)}k(t,t^{\prime})
&R_{t^{\prime}}\otimes_{k(t^{\prime})}k(t,t^{\prime})
\arrow{l}[swap]{\mathcal{O}(\phi_{t^{\prime},t})}\\
R_{t}\otimes_{k}k(t^{\prime})\arrow{u}
&R_{t^{\prime}}\otimes_{k}k(t).\arrow{u}
\arrow[dashed]{l}[swap]{\mathcal{O}(\phi)}
\end{tikzcd}
\]
On replacing $t^{\prime}$ with $t$, we get a homomorphism $\mathcal{O(}{}%
\phi)\colon R_{t}\otimes_{k}k(t)\rightarrow k(t)\otimes_{k}R_{t}$ satisfying
the condition $\mathcal{O(}{}\phi)_{2}=\mathcal{O(}{}\phi)_{1}\circ
\mathcal{O(}{}\phi)_{3}$ of Theorem \ref{dt23}. Thus, there exists a
$k$-algebra $R$ and an isomorphism $\mathcal{O}{}(\varphi)\colon
k(t)\otimes_{k}R\rightarrow R_{t}$ such that
\[
\mathcal{O(}{}\phi)=(\id_{B}\otimes\mathcal{O(}{}\varphi))\circ(\mathcal{O(}%
{}\varphi)\otimes\id_{B})^{-1}.
\]
Now $(\Spec(R),\varphi)$ is the requred pair.

In the general case, there is a commutative diagram%
\[
\begin{tikzcd}
V_{t}\times_{\Spec k(t)}\Spec k(t,t^{\prime})
\arrow{r}{\phi_{t^{\prime},t}}\arrow{d}
&V_{t^{\prime}}\times_{\Spec k(t^{\prime})}\Spec k(t,t^{\prime})\arrow{d}\\
V_{t}\times_{\Spec k}\Spec k(t^{\prime})\arrow[dashed]{r}{\phi}
&V_{t^{\prime}}\times_{\Spec k}\Spec k(t).
\end{tikzcd}
\]
This case follows from Theorem \ref{dt24}.
\end{proof}

\begin{theorem}
\label{dt25}Let $\Omega$ be an algebraically closed field of infinite
transcendence degree over a perfect field $k$. Then descent is effective for
quasi-projective schemes over $\Omega$.
\end{theorem}

\begin{proof}
Let $(\varphi_{\sigma})_{\sigma}$ be a descent datum on an algebraic 
scheme $V$ over $\Omega$. Because $(\varphi_{\sigma})_{\sigma}$ is continuous, it is
split by a model of $V$ over some subfield $K$ of $\Omega$ finitely generated
over $k$. Let $k^{\prime}$ be the algebraic closure of $k$ in $K$; then
$k^{\prime}$ is a finite extension of $k$ and $K$ is a regular extension of
$k$. Write $K=k(t)$, and let $(V_{t},\varphi^{\prime})$ be a model of $V$ over
$k(t)$ splitting $(\varphi_{\sigma})$. According to Lemma \ref{dt17}, there
exists a $\sigma\in\Aut(\Omega/k)$ such that $k(t^{\prime})\overset{\df}{=}%
\sigma k(t)$ and $k(t)$ are linearly disjoint over $k$. The isomorphism%
\[
V_{t\Omega}\overset{\varphi^{\prime}}{\longrightarrow}V\overset{\varphi
_{\sigma}^{-1}}{\longrightarrow}\sigma V\overset{(\sigma\varphi^{\prime}%
)^{-1}}{\longrightarrow}V_{t^{\prime},\Omega}%
\]
is defined over $k(t,t^{\prime})$ and satisfies the conditions of Theorem
\ref{dt25b}. Therefore, there exists a model $(W,\varphi)$ of $V$ over
$k^{\prime}$ splitting $(\varphi_{\sigma})_{\sigma\in\Aut(\Omega/k(t)}$.

For $\sigma,\tau\in\Aut(\Omega/k)$, let $\varphi_{\tau,\sigma}$ be the
composite of the isomorphisms%
\[
\sigma W\overset{\sigma\varphi}{\longrightarrow}\sigma V\overset{\varphi
_{\sigma}}{\longrightarrow}V\overset{\varphi_{\tau}^{-1}}{\longrightarrow}\tau
V\overset{\tau\varphi}{\longrightarrow}\tau W\text{.}%
\]
Then $\varphi_{\tau,\sigma}$ is defined over the algebraic closure of $k$ in
$\Omega$ and satisfies the conditions of Theorem \ref{dt25a}, which gives a
model of $W$ over $k$ splitting $(\varphi_{\sigma})_{\sigma\in\Aut(\Omega
/k)}.$
\end{proof}

\begin{notes}
\cite{weil1956} is the first important paper in descent theory. Its results
were not superseded by the results of Grothendieck. As noted the statements of
Theorems \ref{dt25a} and \ref{dt25b} are from Weil's paper. Their proofs are
probably also Weil's. Theorem \ref{dt25} is Theorem 1.1. of \cite{milne1999}.
\end{notes}

\bsmall
\bibliographystyle{cbe}
\bibliography{DTrefs.bib}
\esmall

\end{document}